\documentclass[letterpaper,reqno,10pt,twoside]{amsart}
\usepackage{amsmath,amsthm,amsfonts,amssymb,euscript,mathrsfs,graphics,color,latexsym,marginnote}
\usepackage[dvips]{graphicx}
\usepackage[margin=1in]{geometry}
\usepackage{hyperref}
\usepackage[toc,page]{appendix}
\usepackage{relsize}
\usepackage[shortlabels]{enumitem}
\usepackage[toc,page]{appendix}
\usepackage[all]{xy}
\usepackage{dsfont}

\usepackage[dvipsnames]{xcolor}

\usepackage{hyperref}
\hypersetup{colorlinks=true, pdfstartview=FitV,linkcolor=blue!70!black,citecolor=red!70!black, urlcolor=green!60!black}
\definecolor{labelkey}{rgb}{0.6,0,0}

\newtheorem{theorem}{Theorem}[section]

\newtheorem{lemma}[theorem]{Lemma}
\newtheorem{proposition}[theorem]{Proposition}

\theoremstyle{definition}

\theoremstyle{remark}
\newtheorem{remark}{Remark}[section]

\makeatletter
\renewcommand \theequation {%
\ifnum \c@section>\z@ \@arabic\c@section.%
\fi\@arabic\c@equation} \@addtoreset{equation}{section}
\makeatother

\makeatletter
\renewcommand \theequation {%
\ifnum \c@section>\z@ \@arabic\c@section.%
\fi\@arabic\c@equation} \@addtoreset{equation}{section}
\makeatother

\makeatletter
\DeclareFontFamily{OMX}{MnSymbolE}{}
\DeclareSymbolFont{MnLargeSymbols}{OMX}{MnSymbolE}{m}{n}
\SetSymbolFont{MnLargeSymbols}{bold}{OMX}{MnSymbolE}{b}{n}
\DeclareFontShape{OMX}{MnSymbolE}{m}{n}{
    <-6>  MnSymbolE5
   <6-7>  MnSymbolE6
   <7-8>  MnSymbolE7
   <8-9>  MnSymbolE8
   <9-10> MnSymbolE9
  <10-12> MnSymbolE10
  <12->   MnSymbolE12
}{}
\DeclareFontShape{OMX}{MnSymbolE}{b}{n}{
    <-6>  MnSymbolE-Bold5
   <6-7>  MnSymbolE-Bold6
   <7-8>  MnSymbolE-Bold7
   <8-9>  MnSymbolE-Bold8
   <9-10> MnSymbolE-Bold9
  <10-12> MnSymbolE-Bold10
  <12->   MnSymbolE-Bold12
}{}

\let\llangle\@undefined
\let\rrangle\@undefined
\DeclareMathDelimiter{\llangle}{\mathopen}%
    {MnLargeSymbols}{'164}{MnLargeSymbols}{'164}
\DeclareMathDelimiter{\rrangle}{\mathclose}%
    {MnLargeSymbols}{'171}{MnLargeSymbols}{'171}
\makeatother

\DeclareMathOperator*{\esssup}{ess\,sup}

\providecommand{\abs}[1]{\left\vert#1\right\vert}
\providecommand{\babs}[1]{\big\vert#1\big\vert}
\providecommand{\Babs}[1]{\Big\vert#1\Big\vert}
\providecommand{\nm}[1]{\left\Vert#1\right\Vert}
\providecommand{\bnm}[1]{\big\Vert#1\big\Vert}

\providecommand{\nnm}[1]{\left\vert\kern-0.25ex\left\vert\kern-0.25ex\left\vert#1\right\vert\kern-0.25ex\right\vert\kern-0.25ex\right\vert}
\providecommand{\onnm}[1]{\vert\kern-0.25ex\vert\kern-0.25ex\vert#1\vert\kern-0.25ex\vert\kern-0.25ex\vert}
\providecommand{\bnnm}[1]{\big\vert\kern-0.25ex\big\vert\kern-0.25ex\big\vert#1\big\vert\kern-0.25ex\big\vert\kern-0.25ex\big\vert}
\providecommand{\Bnnm}[1]{\Big\vert\kern-0.25ex\Big\vert\kern-0.25ex\Big\vert#1\Big\vert\kern-0.25ex\Big\vert\kern-0.25ex\Big\vert}

\providecommand{\br}[1]{\left\langle #1 \right\rangle}
\providecommand{\bbr}[1]{\big\langle #1 \big\rangle}

\providecommand{\brb}[2]{\left\langle #1 \right\rangle_{#2}}
\providecommand{\bbrb}[2]{\big\langle #1 \big\rangle_{#2}}

\providecommand{\tnm}[1]{\left\Vert#1\right\Vert_{L^2}}
\providecommand{\btnm}[1]{\big\Vert#1\big\Vert_{L^2}}
\providecommand{\lnm}[1]{\left\Vert#1\right\Vert_{L^{\infty}}}
\providecommand{\tnms}[2]{\left\vert#1\right\vert_{L^2_{#2}}}

\providecommand{\lnms}[2]{\left\vert#1\right\vert_{L^{\infty}_{#2}}}

\providecommand{\dbr}[1]{\left\llangle #1 \right\rrangle}
\providecommand{\dobr}[1]{\llangle #1 \rrangle}
\providecommand{\dbbr}[1]{\big\llangle #1 \big\rrangle}
\providecommand{\dBbr}[1]{\Big\llangle #1 \Big\rrangle}

\providecommand{\pnnm}[2]{\left\vert\kern-0.25ex\left\vert\kern-0.25ex\left\vert#1\right\vert\kern-0.25ex\right\vert\kern-0.25ex\right\vert_{L^{#2}}}
\providecommand{\tnnm}[1]{\left\vert\kern-0.25ex\left\vert\kern-0.25ex\left\vert#1\right\vert\kern-0.25ex\right\vert\kern-0.25ex\right\vert_{L^{2}}}
\providecommand{\otnnm}[1]{\vert\kern-0.25ex\vert\kern-0.25ex\vert#1\vert\kern-0.25ex\vert\kern-0.25ex\vert_{L^{2}}}
\providecommand{\btnnm}[1]{\big\vert\kern-0.25ex\big\vert\kern-0.25ex\big\vert#1\big\vert\kern-0.25ex\big\vert\kern-0.25ex\big\vert_{L^{2}}}
\providecommand{\Btnnm}[1]{\Big\vert\kern-0.25ex\Big\vert\kern-0.25ex\Big\vert#1\Big\vert\kern-0.25ex\Big\vert\kern-0.25ex\Big\vert_{L^{2}}}
\providecommand{\lnnm}[1]{\left\vert\kern-0.25ex\left\vert\kern-0.25ex\left\vert#1\right\vert\kern-0.25ex\right\vert\kern-0.25ex\right\vert_{L^{\infty}}}
\providecommand{\tnnms}[2]{\left\Vert#1\right\Vert_{L^2_{#2}}}
\providecommand{\btnnms}[2]{\big\Vert#1\big\Vert_{L^2_{#2}}}
\providecommand{\lnnms}[2]{\left\Vert#1\right\Vert_{L^{\infty}_{#2}}}

\def\ud{\mathrm{d}}
\def\dt{\partial_t}
\def\p{\partial}
\def\ls{\lesssim}
\def\gs{\gtrsim}

\def\rt{\rightarrow}
\def\r{\mathbb{R}}
\def\no{\nonumber}
\def\ue{\mathrm{e}}

\def\ds{\displaystyle}

\def\S{\mathbb{S}}

\def\nx{\nabla_x}
\def\dx{\Delta_x}

\def\s{\S}
\def\e{\varepsilon}

\def\vw{w}
\def\vx{x}
\def\vn{n}

\def\u{U}
\def\tu{\widetilde{U}}
\def\bu{\overline{\u}}
\def\uu{U^B}
\def\tuu{\widetilde{U}^B}

\def\uuu{U^I}

\def\re{\mathcal{R}}
\def\bre{\overline{\re}}
\def\ire{\re-\bre}
\def\ss{\mathcal{S}}
\def\g{\mathcal{G}}
\def\ge{E[h]}
\def\h{\mathcal{H}}

\def\vr{r}
\def\vt{\varsigma}

\def\bl{\uu}
\def\bbl{\overline{\uu}}
\def\ch{\widetilde{\chi}}

\def\rp{\r^+}
\def\oo{\delta}

\def\test{\xi}
\def\Test{\zeta}

\def\bl{\Phi}
\def\bbl{\overline{\bl}}
\def\bll{\Psi}
\def\bbll{\overline{\bll}}
\def\il{\Theta}
\def\iil{\overline{\il}}

\def\vr{\mathbf{r}}
\def\mn{\mu}

\def\pl{L}
\def\ua{u_a}
\def\bua{\overline{u}_a}
\def\weak{\varphi}
\def\z{\mathcal{I}}

\def\NN{p}

\def\rc{R}
\def\pp{\mathscr{P}_{\gamma}}
\def\rr{\mathscr{R}}
\def\sr{\mathscr{R}_{\gamma}}

\def\sss{\ss^{I\!S}}
\def\ssa{\ss^{B\!L}_1}
\def\ssb{\ss^{B\!L}_2}
\def\ssc{\ss^{B\!L}_3}
\def\ssd{\ss^{I\!L}}

\begin{document}

\title{Diffusive Limit of the Unsteady Neutron Transport Equation in Bounded Domains}

\author[Z. Ouyang]{Zhimeng Ouyang}
\address[Z. Ouyang]{
   \newline\indent Department of Mathematics, University of Chicago}
\email{ouyangzm9386@uchicago.edu}
\thanks{Z. Ouyang is supported by NSF Grant DMS-2202824.}

\date{}



\maketitle

\begin{abstract}
The justification of hydrodynamic limits in non-convex domains has long been an open problem due to the singularity at the grazing set. In this paper, we investigate the unsteady neutron transport equation in a general bounded domain with the in-flow, diffuse-reflection, or specular-reflection boundary condition. 
Using a novel kernel estimate, we demonstrate the optimal $L^2$ diffusive limit in the presence of both initial and boundary layers.
Previously, this result was only proved for convex domains when the time variable is involved.
Our approach is highly robust, making it applicable to all basic types of physical boundary conditions.
\end{abstract}

\pagestyle{myheadings} \thispagestyle{plain} \markboth{Z. OUYANG}{DIFFUSIVE LIMIT OF UNSTEADY NEUTRON TRANSPORT EQUATION IN BOUNDED DOMAINS}

\smallskip
\setcounter{tocdepth}{1}
\tableofcontents

%%%%%%%%%%%%%%%%%%%%%%%%%%%%%%%%%%%%%%%%%%%%%%%%%%%%%%%%%%%%%%%%%%%%%%%%%%%%%%%%%%
\section{Introduction}\label{Sec:intro}
%%%%%%%%%%%%%%%%%%%%%%%%%%%%%%%%%%%%%%%%%%%%%%%%%%%%%%%%%%%%%%%%%%%%%%%%%%%%%%%%%%

The neutron transport equation is a fundamental model for studying particle dynamics in confined spaces, such as nuclear reactors and medical imaging devices (e.g., CT, MRI). Its investigation dates back to the dawn of the atomic age and has motivated significant developments in applied mathematics, such as the discontinuous Galerkin (DG) method \cite{Cockburn.Shu1989, Cockburn.Hou.Shu1990, Cockburn.Karniadakis.Shu2000} and the inverse transport problem \cite{Bal2009, Bal.Ren2012}. 

We consider the unsteady neutron transport equation in a
three-dimensional smooth bounded domain $\Omega$ without assuming any convexity. 
Denote the time variable $t\in\rp$, the space variable $\vx=(x_1,x_2,x_3)\in\Omega$, and the velocity variable
$\vw=(w_1,w_2,w_3)\in\s^2$. 
Let the neutron density $u^{\e}(t,\vx,\vw)$ satisfy the in-flow boundary problem
\begin{align}\label{transport}
\left\{
\begin{array}{l}\displaystyle
\e\dt u^{\e}+\vw\cdot\nabla_x u^{\e}+\e^{-1}\big(u^{\e}-\overline{u}^{\e}\big)=0\ \ \text{ in}\ \ \rp\times\Omega\times\s^2,\\\rule{0ex}{1.5em}
u^{\e}(0,\vx,\vw)=u_o(\vx,\vw)\ \ \text{ in}\ \  \Omega\times\s^2,\\\rule{0ex}{1.5em}
u^{\e}(t,\vx_0,\vw)=g(t,\vx_0,\vw)\ \ \text{ for}\
\ \vx_0\in\p\Omega\ \ \text{and}\ \ \vw\cdot\vn(\vx_0)<0.
\end{array}
\right.
\end{align}
Here, $\overline{u}^{\e}$ is the velocity average of $u^{\e}$, i.e.,
\begin{align}\label{average}
\overline{u}^{\e}(t,\vx)=\frac{1}{4\pi}\int_{\s^2}u^{\e}(t,\vx,\vw)\,\ud{\vw}.
\end{align}
Throughout the paper, we will use $\overline{f}(t,x)$ to denote the velocity average of a function $f(t,x,w)$.
$u_o$ and $g$ are the given initial data and boundary data, respectively,
and $\vn(\vx_0)$ is the outward unit normal vector at boundary point $\vx_0\in\p\Omega$.  
Our goal is to study the asymptotic behavior of the solution $u^{\e}$ as the Knudsen number $\e\rt0^+$.

When the incoming particles are distributed like a probability average of the outgoing particles (possibly plus a small perturbation), then the boundary condition in \eqref{transport} is replaced by 
\begin{align}\label{diffuse-BC}
    u^{\e}(t,\vx_0,\vw)=\pp[u^{\e}](t,\vx_0)+\e h(t,\vx_0,\vw)\ \ \text{ for}\
\ \vx_0\in\p\Omega\ \ \text{and}\ \ \vw\cdot\vn(\vx_0)<0,
\end{align}
where 
\begin{align}
    \pp[u^{\e}](t,\vx_0):=c_{\gamma}\int_{\vw\cdot\vn>0}u^{\e}(t,\vx_0,\vw)(\vw\cdot\vn)\,\ud\vw
\end{align}
normalized with the constant $c_{\gamma}=\frac{1}{\pi}$ satisfying $c_{\gamma}\int_{\vw\cdot\vn>0}(\vw\cdot\vn)\,\ud\vw=1$.  
This is called the diffuse-reflection boundary condition.

In the specular-reflection boundary problem, where the particles bounce back specularly (with a small perturbation), the boundary condition in \eqref{transport} is replaced by 
\begin{align}\label{specular-BC}
    u^{\e}(t,\vx_0,\vw)=u^{\e}(t,\vx_0,\rr\vw)+\e h(t,\vx_0,\vw)\ \ \text{ for}\
\ \vx_0\in\p\Omega\ \ \text{and}\ \ \vw\cdot\vn(\vx_0)<0
\end{align}
with $\rr\vw:=\vw-2(\vw\cdot\vn)\vn$.

For both reflection type boundary conditions, we additionally assume the compatibility condition 
\begin{align}\label{g-compatibility=}
    \int_{\vw\cdot\vn<0}h(t,\vx_0,\vw)(\vw\cdot\vn)\,\ud\vw=0\ \ \text{ for all}\ \  \vx_0\in\p\Omega\ \ \text{and}\ \ t\in\rp
\end{align}
on the given perturbative boundary data $h$, 
so that the solution enjoys the null flux condition at the boundary: 
\begin{align}\label{ue-zero-flux}
    \int_{\s^2}u^{\e}(t,\vx_0,\vw)(\vw\cdot\vn)\,\ud\vw = 0 
    \ \ \text{ for all}\ \  \vx_0\in\p\Omega\ \ \text{and}\ \ t\in\rp.
\end{align}

%%%%%%%%%%%%%%%%%%%%%%%%%%%%%%%%%%%%%%%%%%%%%%%%%%%%%%%%%%%%%%%%%%%%%%%%%%%%%%%%%%
\subsection{Notation}
%%%%%%%%%%%%%%%%%%%%%%%%%%%%%%%%%%%%%%%%%%%%%%%%%%%%%%%%%%%%%%%%%%%%%%%%%%%%%%%%%%

Based on the flow direction, we can divide the phase boundary $\gamma:=\p\Omega\times\s^2$ into the incoming boundary $\gamma_-:=\big\{(\vx_0,\vw)\in\gamma:\vw\cdot\vn(\vx_0)<0\big\}$, the outgoing boundary $\gamma_+:=\big\{(\vx_0,\vw)\in\gamma:\vw\cdot\vn(\vx_0)>0\big\}$, and the grazing set $\gamma_0:=\big\{(\vx_0,\vw)\in\gamma:\vw\cdot\vn(\vx_0)=0\big\}$. 
In particular, the boundary conditions of \eqref{transport}, \eqref{diffuse-BC} and \eqref{specular-BC} are only prescribed on $\gamma_{-}$. 
Denote $\Gamma:=[0,T]\times\gamma$ and we may also split $\Gamma=\Gamma_-\cup\Gamma_+\cup\Gamma_0$ in a parallel manner.

Let $\brb{\ \cdot\ ,\ \cdot\ }{\sigma}$ denote the inner product with respect to variable(s) $\sigma$ which can be $t,x,w$, their combinations, or the domain specified, and let $\dbr{\ \cdot\ ,\ \cdot\ }$ be the inner product over $[0,t]\times\Omega\times\s^2$ for $t\in[0,T]$, i.e.,
\begin{align}
    \dbr{f,g} := \int_0^t\!\iint_{\Omega\times\s^2} fg \,\ud\vw\ud\vx\ud t' .
\end{align}
Also, let $\br{\ \cdot\ ,\ \cdot\ }_{\gamma_{\pm}}$ denote the inner product on $\gamma_{\pm}$ with measure $\ud\gamma:=\abs{w\cdot n}\ud w\ud S_x$, where $\ud S_x$ is the differential of the boundary surface. 
When the time variable is involved, we use $\br{\ \cdot\ ,\ \cdot\ }_{\Gamma_{\pm}}$ 
to denote
\begin{align}
    \br{f,g}_{\Gamma_{\pm}} := \int_0^t\!\iint_{\p\Omega\times\s^2} fg \abs{w\cdot n}\ud w\ud S_x\ud t' .
\end{align}

Define the bulk and boundary $L^2$ norms
\begin{align}
    \tnm{f}:=\left(\iint_{\Omega\times\s^2}\babs{f(x,w)}^2\ud w\ud x\right)^{\frac{1}{2}},\qquad \tnms{f}{\gamma_{\pm}}:=\bigg(\int_{\gamma_{\pm}}\babs{f(x,w)}^2\ud\gamma\bigg)^{\frac{1}{2}},
\end{align}
and the $L^{\infty}$ norms
\begin{align}
    \lnm{f}:=\esssup_{(x,w)\in\Omega\times\s^2}\babs{f(x,w)},\qquad
    \lnms{f}{\gamma_{\pm}}:=\esssup_{(x,w)\in\gamma_{\pm}}\babs{f(x,w)}.
\end{align}
When the $t$ variable is involved, we use $\tnnm{f}$, $\tnnms{f}{\Gamma_\pm}$, $\lnnm{f}$ and $\lnnms{f}{\Gamma_\pm}$ to denote the corresponding norms.

For $k,m,l\in\mathbb{N}$, $1\leq p,q,r\leq\infty$, let $W^{k,p}$ be the Sobolev space, and let $\nnm{\,\cdot\,}_{W^{l,r}_tW^{k,p}_xW^{m,q}_w}$ denote the usual Sobolev norm for $t\in\rp$, $x\in\Omega$, and $w\in\s^2$.
If we use mixed norms on the boundary, then denote by $\abs{\,\cdot\,}_{W^{k,p}_xW^{m,q}_w(\gamma_\pm)}$ the mixed Sobolev norm for $(x,w)\in\gamma_\pm$ with measure $\ud\gamma=\abs{w\cdot n}\ud w\ud S_x$. The same convention also applies to $\nm{\,\cdot\,}_{W^{l,r}_tW^{k,p}_xW^{m,q}_w(\Gamma_\pm)}$ when $t$ is involved.
Sometimes we will omit the variables in the subscript when there is no possibility of confusion.

We will abuse notation for the functions under change of variables: 
for example, if $\mathcal{T}:(t,x,w)\mapsto(\tau,y,v)$, and let $f(t,x,w)=f\circ\mathcal{T}^{-1}(\tau,y,v)=:\widetilde{f}(\tau,y,v)$,
then we just write $f(\tau,y,v)$ instead of $\widetilde{f}(\tau,y,v)$ for notational simplicity.
However, unless otherwise specified, the integrals in the $L^p$-norms and the inner products are still with respect to the original variables $t,x,w$.

Throughout this paper, $C>0$ denotes a constant that only depends on
the spatial domain $\Omega$ or temporal domain $[0,T]$ for some given $T>0$, but does not depend on the data or $\e$. 
It is referred as universal constant and can change from one inequality to another. We write $a\ls b$ to denote $a\leq Cb$ and $a\gs b$ to denote $a\geq Cb$. Also, we write $a\simeq b$ if $a\ls b$ and $a\gs b$. 
We will use $0<\oo\ll 1$ to denote a sufficiently small constant independent of the data or $\e$.

%%%%%%%%%%%%%%%%%%%%%%%%%%%%%%%%%%%%%%%%%%%%%%%%%%%%%%%%%%%%%%%%%%%%%%%%%%%%%%%%%%
\subsection{Main Results}
%%%%%%%%%%%%%%%%%%%%%%%%%%%%%%%%%%%%%%%%%%%%%%%%%%%%%%%%%%%%%%%%%%%%%%%%%%%%%%%%%%

Our main theorems establish the diffusive limit of the unsteady neutron transport equation in a bounded domain, under three different types of boundary conditions: in-flow, diffuse-reflection, and specular-reflection.
As the assumption on data and the limit problem differ among the three cases, it is better to state our results separately:

\begin{theorem}[In-flow case]\label{main theorem}
Given $u_o$ and $g$ satisfying
\begin{align}\label{assumption}
    \nm{u_o}_{W^{3,\infty}_xL^{\infty}_w}+\nm{g}_{W^{2,\infty}_tW^{3,\infty}_xW^{1,\infty}_{w}(\Gamma_-)}\ls 1
\end{align}
and the compatibility condition 
\begin{align}\label{compatibility}
    u_o(\vx_0,\vw)=g(0,\vx_0,\vw)=C(x_0)\ \ \text{ for every}\ \ \vx_0\in\p\Omega,
\end{align}
there exists a unique solution $u^{\e}(t,\vx,\vw)\in L^{\infty}\big([0,T]\times\Omega\times\s^2\big)$ to \eqref{transport} for any $T>0$. Moreover, the solution obeys the estimate
\begin{align}\label{main}
\bnnm{u^{\e}-\u_0}_{L^2([0,T]\times\Omega\times\s^2)}\ls\e^{\frac{1}{2}},
\end{align}
where $\u_0(t,x,w)=\overline\u_0(t,\vx)$ satisfies the heat equation with Dirichlet boundary condition:
\begin{align}
\left\{
\begin{array}{l}
\dt\overline\u_0-\Delta_x\overline\u_0=0\ \ \text{ in}\
\ [0,T]\times\Omega,\\\rule{0ex}{1.5em}
\overline\u_0(0,\vx)=\overline{u}_o(\vx)\ \ \text{ in}\
\ \Omega,\\\rule{0ex}{1.5em}
\overline\u_0(t,\vx_0)=\bl_{0,\infty}(t;\iota_1,\iota_2)\ \ \text{ on}\ \
[0,T]\times\p\Omega,
\end{array}
\right.
\end{align}
in which $\bl_{0,\infty}$ is determined by solving the boundary layer equation (Milne problem) for $\bl_0$:
\begin{align}
\left\{
\begin{array}{l}
\sin\phi\,\dfrac{\p \bl_0 }{\p\eta}+\bl_0 -\bbl_0 =0,\\\rule{0ex}{1.2em}
\bl_0 (t;0,\iota_1,\iota_2;\phi,\psi)=g(t;\iota_1,\iota_2;\phi,\psi)\ \ \text{ for}\ \
\sin\phi>0,\\\rule{0ex}{1.5em}
\ds\lim_{\eta\rt\infty}\bl_0 (t;\eta,\iota_1,\iota_2;\phi,\psi)=\bl_{0,\infty}(t;\iota_1,\iota_2).
\end{array}
\right.
\end{align}
Here, the change of variables $(x;w)\mapsto(\eta,\iota_1,\iota_2;\phi,\psi)$ near the boundary are defined in Section~\ref{sec:boundary-layer}.%}
\end{theorem}

\begin{remark} \label{Rmk:time-decay}
    If we make the stronger assumption on the boundary data that for some $\alpha>0$
    \begin{align}\label{assumption.}
    \nm{\ue^{\alpha t}g}_{W^{2,\infty}_tW^{3,\infty}_xW^{1,\infty}_{w}(\Gamma_-)}\ls 1,
    \end{align} 
    then based on Theorem~\ref{thm:remainder-est}, we have the approximation estimate with time decay for some $0<\beta<\alpha$
    \begin{align}\label{main=}
    \nnm{\ue^{\beta t}\big(u^{\e}-\u_0\big)}_{L^2(\rp\times\Omega\times\s^2)}\ls\e^{\frac{1}{2}}.
    \end{align}
\end{remark}

\begin{remark}
    The estimate \eqref{main} achieves the optimal convergence rate in $L^2_{t,x,w}$. In \cite{AA005,AA012,AA016} for 2D/3D convex domains, it is justified that 
    \begin{align}
    \btnnm{u^{\e}-\tu_0-\tuu_0-\uuu_0}\ls\e^{\frac{5}{6}-},
    \end{align}
    where $\tuu_0$ is the boundary layer with geometric correction, $\uuu_0$ is the initial layer, and $\tu_0$ is the corresponding interior solution. \cite[Theorem 2.1]{Li.Lu.Sun2017} reveals that the two interior solutions differ by
    \begin{align}
        \btnnm{\tu_0-\u_0}\ls\e^{\frac{2}{3}}.
    \end{align}
    Due to the rescaling in the normal spatial variable $\eta=\e^{-1}\mn$, for general non-constant in-flow boundary data, the boundary layer $\tuu_0\neq0$ satisfies
    \begin{align}
        \btnnm{\tuu_0}\simeq\e^{\frac{1}{2}}.
    \end{align}
    Also, 
    due to the rescaling in the temporal variable $\tau=\e^{-2}t$, we have
    \begin{align}
        \btnnm{\uuu_0}\simeq\e.
    \end{align}
    Hence, we conclude that
    \begin{align}
        \btnnm{u^{\e}-\u_0}\simeq\e^{\frac{1}{2}}.
    \end{align}
    This suggests the optimality of our diffusive approximation \eqref{main}.
\end{remark}

In the case of the diffuse and specular reflection boundary conditions, we have the following analogues:

\begin{theorem}[Diffuse-reflection case]\label{main theorem-2}
If we consider \eqref{transport} with the boundary condition replaced by \eqref{diffuse-BC}, 
then under the assumption 
\begin{align}\label{assumption-2}
    \nm{u_o}_{W^{3,\infty}_xL^{\infty}_w}+\nm{h}_{W^{1,\infty}_tW^{2,\infty}_xL^{\infty}_w(\Gamma_-)}\ls 1
\end{align}
with the compatibility conditions \eqref{g-compatibility=} and 
\begin{align}\label{compatibility-2}
    u_o(\vx_0,\vw)=C(x_0), \quad h(0,\vx_0,\vw)=0\ \ \text{ for every}\ \ \vx_0\in\p\Omega,
\end{align}%}
there exists a unique solution $u^{\e}(t,\vx,\vw)\in L^{\infty}\big([0,T]\times\Omega\times\s^2\big)$ for any $T>0$. Moreover, the solution obeys the estimate
\begin{align}\label{main-2}
\bnnm{u^{\e}-\u_0}_{L^2([0,T]\times\Omega\times\s^2)}\ls\e^{\frac{1}{2}},
\end{align}
where $\u_0(t,x,w)=\overline\u_0(t,\vx)$ satisfies the heat equation with Neumann boundary condition:
\begin{align}
\left\{
\begin{array}{l}
\dt\overline\u_0-\Delta_x\overline\u_0=0\ \ \text{ in}\
\ [0,T]\times\Omega,\\\rule{0ex}{1.5em}
\overline\u_0(0,\vx)=\overline{u}_o(\vx)\ \ \text{ in}\
\ \Omega,\\\rule{0ex}{1.5em}
\tfrac{\p\overline\u_0}{\p\vn}(t,\vx_0)=0\ \ \text{ on}\ \
[0,T]\times\p\Omega.
\end{array}
\right.
\end{align}
\end{theorem}

\begin{theorem}[Specular-reflection case]\label{main theorem-3}
If we consider \eqref{transport} with the boundary condition replaced by \eqref{specular-BC},
then under the assumption 
\begin{align}\label{assumption-3}
    \nm{u_o}_{W^{3,\infty}_xL^{\infty}_w}+\nm{h}_{W^{1,\infty}_tW^{2,\infty}_xW^{1,\infty}_w(\Gamma_-)}\ls 1
\end{align}
with the compatibility conditions \eqref{g-compatibility=} and $h|_{\Gamma_0}=0$, as well as
\begin{align}\label{compatibility-3}
    u_o(\vx_0,\vw)=C(x_0), \quad \nx u_o(\vx_0,\vw)=0, \quad h(0,\vx_0,\vw)=0\ \ \text{ for every}\ \ \vx_0\in\p\Omega,
\end{align}
there exists a unique solution $u^{\e}(t,\vx,\vw)\in L^{\infty}\big([0,T]\times\Omega\times\s^2\big)$ for any $T>0$. Moreover, the solution obeys the estimate
\begin{align}\label{main-3}
\bnnm{u^{\e}-\u_0}_{L^2([0,T]\times\Omega\times\s^2)}\ls\e^{\frac{1}{2}},
\end{align}
where $\u_0(t,x,w)=\overline\u_0(t,\vx)$ satisfies the heat equation with Neumann boundary condition:
\begin{align}
\left\{
\begin{array}{l}
\dt\overline\u_0-\Delta_x\overline\u_0=0\ \ \text{ in}\
\ [0,T]\times\Omega,\\\rule{0ex}{1.5em}
\overline\u_0(0,\vx)=\overline{u}_o(\vx)\ \ \text{ in}\
\ \Omega,\\\rule{0ex}{1.5em}
\tfrac{\p\overline\u_0}{\p\vn}(t,\vx_0)=0\ \ \text{ on}\ \
[0,T]\times\p\Omega.
\end{array}
\right.
\end{align}
\end{theorem}

\begin{remark}
To compare the three boundary problems:
In all the three cases,
we achieve the $L^2$ diffusive limit with the same convergence rate
\begin{align}
\bnnm{u^{\e}-\u_0}_{L^2([0,T]\times\Omega\times\s^2)}\ls\e^{\frac{1}{2}},
\end{align}
and in a similar fashion as Remark~\ref{Rmk:time-decay}, the corresponding time decay if assuming decay on the boundary data.
For the in-flow case,
the interior solution $\u_0$ satisfies the heat equation with Dirichlet boundary determined by boundary layers.
For both diffuse and specular cases, $\u_0$ does not depend on boundary layers due to the vanishing leading-order boundary layer.
However, the proof of the specular case still requires introducing a next-order (artificial) ``boundary layer'', in order to enforce the perfect specular boundary condition for the remainder
(because we cannot control the boundary trace of remainder in the energy or kernel estimate); while the diffuse case allows small perturbation that can be controlled by $\big(1-\pp\big)[\re]$ in the energy, and thus no boundary layer in needed in the analysis.
\end{remark}

%%%%%%%%%%%%%%%%%%%%%%%%%%%%%%%%%%%%%%%%%%%%%%%%%%%%%%%%%%%%%%%%%%%%%%%%%%%%%%%%%%
\subsection{Overview of the Method}
%%%%%%%%%%%%%%%%%%%%%%%%%%%%%%%%%%%%%%%%%%%%%%%%%%%%%%%%%%%%%%%%%%%%%%%%%%%%%%%%%%

The neutron transport equation has been investigated from various perspectives since the 1960s. For the physical modeling and formal asymptotic expansion,  we refer to \cite{Larsen1974=, Larsen1974, Larsen1975, Larsen1977, Larsen.D'Arruda1976, Larsen.Habetler1973, Larsen.Keller1974, Larsen.Zweifel1974, Larsen.Zweifel1976}. For theoretical analysis, we refer to \cite{Bensoussan.Lions.Papanicolaou1979, Bardos.Santos.Sentis1984, Bardos.Golse.Perthame1987, Bardos.Golse.Perthame.Sentis1988}. For recent developments on the diffusive limit and the half-space problem, we refer to \cite{AA003, AA016, AA020, Li.Lu.Sun2015, Li.Lu.Sun2015==, Li.Lu.Sun2017} and the references therein.

In bounded domains, the mismatch of boundary conditions for the asymptotic expansion calls for a kinetic correction -- the so-called boundary layer (Knudsen layer). As far as we are aware of, there are two main approaches:

In flat domains (e.g. half space or $\mathbb{T}^2\times[0,1]$), the boundary layer $\uu_0(t;\eta,\iota_1,\iota_2;\phi,\psi)$ relies on the Milne problem (cf. \cite{Bensoussan.Lions.Papanicolaou1979, Bardos.Santos.Sentis1984})
\begin{align}\label{intro 01}
    \sin\phi\,\frac{\p \uu_0}{\p\eta}+\uu_0-\overline{\uu_0}=0.
\end{align}
This formulation is consistent with the intuitive derivation from the transport operator under change of coordinates and rescaling (see \eqref{expand 6}).

Unfortunately, in curved domains, the non-vanishing curvature and geometric effects are non-negligible due to the grazing set singularity. A surprising counter-example is constructed in \cite{AA003} so that 
\begin{align}
    \lim_{\e\rt0}\lnnm{u^{\e}-{\u_0}-{\uu_0}-{\uuu_0}}\neq0.
\end{align}

As \cite{AA003,AA016} pointed out, in curved convex domains, the boundary layer $\tuu_0(t;\eta,\iota_1,\iota_2;\phi,\psi)$ relies on the so-called $\e$-Milne problem with geometric correction: 
\begin{align}\label{intro 02}
    \sin\phi\,\frac{\p\tuu_0}{\p\eta}-\frac{\e}{1-\e\eta}\cos\phi\,\frac{\p\tuu_0}{\p\phi}+\tuu_0-\overline{\tuu_0}=0,
\end{align}
where the extra correction term $-\frac{\e}{1-\e\eta}\cos\phi\frac{\p\tuu_0}{\p\phi}$ helps provide the weighted $W^{1,\infty}$ bound of $\tuu_0$ and eventually leads to the desired diffusive limit.

In both flat and convex domains, the proof relies on expanding the boundary layer to higher order which provides sufficient $\e$ power to close the remainder estimate.
However, in non-convex domains, as \cite{AA006, AA020} reveals, the Milne problem with or without geometric correction does not guarantee $W^{1,\infty}$ regularity of the boundary layer. As a consequence, we cannot even bound the remainder induced by the leading-order boundary layer.

In this paper, we will employ a fresh approach to justify the $L^2$ diffusive expansion in general bounded domains with a cutoff boundary layer and the novel remainder estimate. We intend to show that
\begin{align}\label{intro 03}
    \lim_{\e\rt0}\tnnm{u^{\e}-\u_0-\uu_0-\uuu_0}=0
\end{align}
for the leading-order interior solution $\u_0$, boundary layer $\uu_0$, and initial layer $\uuu_0$.

It is well-known that the classical bounds for the remainder $\re:=u^{\e}-\u_0-\uu_0-\uuu_0$ read (see \cite{AA005,AA012,AA016})
\begin{align}
\tnm{\re(t)}+\e^{-\frac{1}{2}}\tnnms{\re}{\Gamma_+}+\e^{-1}\tnnm{\ire}&\ls\oo\tnnm{\bre}+1,\\
\tnnm{\bre}&\ls \e^{-1}\tnnm{\ire}+\tnnms{\re}{\Gamma_+}+\e^{\frac{1}{2}}\tnm{\re(t)}+1,\label{intro 05}
\end{align}
which yield
\begin{align}\label{intro 04}
    \tnm{\re(t)}+\e^{-\frac{1}{2}}\tnnms{\re}{\Gamma_+}+\e^{-1}\tnnm{\ire}+\tnnm{\bre}\ls1.
\end{align}
Clearly, without expanding to higher-order boundary layers, this is insufficient for the validity of \eqref{intro 03}.

The bottleneck of \eqref{intro 04} lies in the $\bre$ bound \eqref{intro 05}. 
Initiated from \cite{AA020} for the steady problem with the in-flow boundary, we design several delicate test functions in the kernel estimate to obtain  
\begin{align}\label{intro 06}
\tnnm{\bre}\ls \tnnm{\ire}+\tnnms{\re}{\Gamma_+}+\e^{\frac{1}{2}}\tnm{\re(t)}+\e^{\frac{1}{2}},
\end{align}
which will lead to the desired remainder estimate
\begin{align}
    \tnm{\re(t)}+\e^{-\frac{1}{2}}\tnnms{\re}{\Gamma_+}+\e^{-1}\tnnm{\ire}+\e^{-\frac{1}{2}}\tnnm{\bre}\ls1.
\end{align}

The key to achieve \eqref{intro 06} is a tricky combination of three conservation laws. We introduce the auxiliary function $\test(t,\vx)$ satisfying $-\Delta_x\test=\bre$ and $\test|_{\p\Omega}=0$. For the remainder equation
\begin{align}
    \e\dt\re+\vw\cdot\nx\re+\e^{-1}\big(\ire\big)=\ss,
\end{align}
testing against $\e^{-1}\test$ yields
\begin{align}\label{kk 01}
    \dbr{\dt\bre,\test}-\e^{-1}\dbbr{\ire,\vw\cdot\nx\test}=\e^{-1}\dbr{\ss,\test},
\end{align}
the choice of test function $w\cdot\nx\test$ yields
\begin{align}\label{kk 02}
    \e\dbbr{\dt\re,w\cdot\nx\test}+\bbr{\re,w\cdot\nx\test}_{\Gamma_+} - \bbr{\g, w\cdot\nx\test}_{\Gamma_-} \\
    -\dbbr{\re,\vw\cdot\nx\big(w\cdot\nx\test\big)}+\e^{-1}\dbbr{\ire,w\cdot\nx\test}
    &=\dbbr{\ss,w\cdot\nx\test},\no
\end{align}
while the test function $\e\dt\test$ leads to
\begin{align}\label{kk 03}
    \e^2\dbbr{\dt\bre,\dt\test}-\e\dbbr{\ire,\vw\cdot\nx\dt\test}=\e\dbbr{\ss,\dt\test}.
\end{align}
On the one hand, the linear combination of \eqref{kk 01} and \eqref{kk 02} crucially cancels $\e^{-1}\dbbr{\ire,w\cdot\nx\test}$ so that it kills the worst contribution of $\e^{-1}\tnnm{\ire}$, and also provides the control of $-\dbbr{\bre,\vw\cdot\nx\big(w\cdot\nx\test\big)}\simeq\tnnm{\bre}^2$. 
On the other hand, the estimate of $\e\dbbr{\dt\re,w\cdot\nx\test}\sim \e\dbbr{\re,w\cdot\nx\dt\test}$ (up to some good terms) calls for the control of $\dt\nx\test$, which is in turn provided by \eqref{kk 03} as $\e^2\dbbr{\dt\bre,\dt\test}=\e^2\btnnm{\nx\dt\test}^2$.

Combined with the key favorable sign of $\dbbr{\dt\bre,\test}\sim\btnm{\nx\test(t)}^2$ and a careful analysis of the source terms $\e^{-1}\dbr{\ss,\test}$, $\dbbr{\ss,w\cdot\nx\test}$, and $\e\dbbr{\ss,\dt\test}$ using Hardy's inequality and delicately chosen norms, we are able to create an extra gain of $\e^{\frac{1}{2}}$ for the kernel bound.  We then conclude the remainder estimate without any further expansion of higher-order boundary layers.

The remaining of this paper is structured as follows: Section~\ref{Sec:asymptotic} focuses on the asymptotic expansion for the interior solution, boundary and initial layers. Section~\ref{Sec:remainder-eq} presents the setup of the remainder equation, and Section~\ref{Sec:remainder-est} elaborates the estimates of $\ire$ and $\bre$. 
Finally, the main theorem for the inflow case is proved in Section~\ref{Sec:diffusive-limit}, while the diffuse and specular boundary problems are discussed in Section~\ref{Sec:diffuse-BC} and Section~\ref{Sec:specular-BC}, respectively.

\bigskip
%%%%%%%%%%%%%%%%%%%%%%%%%%%%%%%%%%%%%%%%%%%%%%%%%%%%%%%%%%%%%%%%%%%%%%%%%%%%%%%%%%
\section{Asymptotic Analysis}\label{Sec:asymptotic}
%%%%%%%%%%%%%%%%%%%%%%%%%%%%%%%%%%%%%%%%%%%%%%%%%%%%%%%%%%%%%%%%%%%%%%%%%%%%%%%%%%

We seek a solution to \eqref{transport} in the form of 
\begin{align}\label{expand}
u^{\e}=\u+\uuu+\uu+\re=\left(\u_0+\e\u_1+\e^2\u_2\right)+\left(\uuu_0+\e\uuu_1\right)+\uu_0+\re,
\end{align}
where the interior solution 
\begin{align}\label{expand 1}
\u(t,x,w):= \u_0(t,x,w)+\e\u_1(t,x,w)+\e^2\u_2(t,x,w),
\end{align}
the initial layer 
\begin{align}\label{expand 2'}
\uuu(\tau,x,w):= \uuu_0(\tau,x,w)+\e\uuu_1(\tau,x,w),
\end{align}
and the boundary layer 
\begin{align}\label{expand 2}
\uu(t;\eta,\iota_1,\iota_2;\phi,\psi):= \uu_0(t;\eta,\iota_1,\iota_2;\phi,\psi).
\end{align}
Here, $\u_0$, $\u_1$, $\u_2$, $\uuu_0$, $\uuu_1$, and $\uu_0$ will be constructed in the following subsections, and $\re(t,x,w)$ is the remainder.

%%%%%%%%%%%%%%%%%%%%%%%%%%%%%%%%%%%%%%%%%%%%%%%%%%%%%%%%%%%%%%%%%%%%%%%%%%%%%%%%%%
\subsection{Interior Solution}\label{sec:interior}
%%%%%%%%%%%%%%%%%%%%%%%%%%%%%%%%%%%%%%%%%%%%%%%%%%%%%%%%%%%%%%%%%%%%%%%%%%%%%%%%%%

Inserting the asymptotic expansion ansatz $u^{\e}\sim \sum_{k=0}^{\infty}\u_k\,\e^k$ into \eqref{transport}, we get a hierarchy of equations by looking at each order of $\e$. 
Following the analysis in \cite{AA016}, we deduce that
\begin{align}
&\u_0=\bu_0,\qquad
\dt\bu_0-\Delta_x\bu_0=0,\label{expand 4}\\ 
&\u_1=\bu_1-\vw\cdot\nx\u_{0},\qquad
\dt\bu_1-\Delta_x\bu_1=0,\label{expand 4'}\\ 
&\u_2=\bu_2-\vw\cdot\nx\u_{1}-\dt\u_0,\qquad
\dt\bu_2-\Delta_x\bu_2=0. \label{expand 4''}
\end{align}
The initial and boundary conditions for $\bu_0$, $\bu_1$, and $\bu_2$ are determined by the initial and boundary layers.

%%%%%%%%%%%%%%%%%%%%%%%%%%%%%%%%%%%%%%%%%%%%%%%%%%%%%%%%%%%%%%%%%%%%%%%%
\subsection{Initial Layer}\label{sec:initial-layer}
%%%%%%%%%%%%%%%%%%%%%%%%%%%%%%%%%%%%%%%%%%%%%%%%%%%%%%%%%%%%%%%%%%%%%%%%

Based on the principle of dominant balance, 
we find the correct scaling for the time variable is $\tau=\e^{-2}t$, and so $\p_t=\e^{-2}\p_{\tau}$. 
Under the substitution $t\mapsto\tau$, the rescaled equation of \eqref{transport} for the initial layer reads
\begin{align}\label{initial equation}
    \p_\tau \uuu+\e\vw\cdot\nabla_x \uuu+\big(\uuu-\overline{\uuu}\big)=0.
\end{align}

Inserting the expansion $\uuu\sim \sum_{k=0}^{\infty}\uuu_k\,\e^k$ into \eqref{initial equation} and comparing the order of $\e$, we find that
\begin{align}
    \p_{\tau}\uuu_0+\uuu_0-\overline{\uuu_0}&=0,\\
    \p_{\tau}\uuu_1+\uuu_1-\overline{\uuu_1}&=-\vw\cdot\nx\uuu_0.
\end{align}

Let us consider the general initial layer problem 
\begin{align} \label{Pb:initial-layer}
        \left\{
        \begin{array}{l}
        \dfrac{\ud\il}{\ud\tau}+\il-\iil=S,\\\rule{0ex}{1.5em}
        \il(0,\vx,\vw)=\il_o(\vx,\vw).
        \end{array}
        \right.
    \end{align}
Here, $S(\tau,x,w)$ is the forcing term, and $\il_o(\vx,\vw)$ is the given initial data.
We are interested in the solution $\il(\tau,x,w)$ that satisfies
\begin{align} \label{initial-layer-decay}
    \lim_{\tau\rt\infty}\il(\tau,\vx,\vw)=\il_{\infty}(\vx)
\end{align}
for some $\il_{\infty}(\vx)$ that does not depend on $w$.

The following proposition guarantees the solvability and regularity of the above initial layer problem: 

\begin{proposition}[Initial layer problem] \label{prop:initial-wellposedness}
Let $k\in\mathbb{N}$.
    Assume 
    \begin{align}
        \nm{\il_o}_{W^{k,\infty}_xL^{\infty}_w} +\nnm{\ue^{\alpha\tau}S}_{L^{\infty}_\tau W^{k,\infty}_xL^{\infty}_w} \ls 1
    \end{align}
    for some $\alpha>0$.
    Then there exist a unique solution $\il(\tau,x,w)\in L^{\infty}_\tau W^{k,\infty}_xL^{\infty}_w$ to \eqref{Pb:initial-layer} and a $\il_{\infty}(\vx)\in W^{k,\infty}_x$ such that \eqref{initial-layer-decay} holds in the sense of 
    \begin{align} \label{initial-layer-decay'}
    \nm{\il_{\infty}}_{W^{k,\infty}_x}+\nnm{\ue^{\beta\tau}(\il-\il_{\infty})}_{L^{\infty}_\tau W^{k,\infty}_xL^{\infty}_w}\ls 1
    \end{align}
    for any $0<\beta\leq \min\{1,\alpha\}$.
\end{proposition}

\begin{proof}

Based on ODE theory, the initial value problem \eqref{Pb:initial-layer} admits a unique solution.
We first integrate the equation in $w$ to find $\p_\tau\overline{\il}=\overline{S}$ 
and so $\iil = \overline{\il}_o + \int_0^\tau \overline{S}$.
Then subtracting this from \eqref{Pb:initial-layer}, we get
\begin{align} \label{Pb:initial-layer'}
        \left\{
        \begin{array}{l}
        \dfrac{\ud}{\ud\tau}\big(\il-\overline{\il}\big)+\big(\il-\overline{\il}\big)=S-\overline{S},\\\rule{0ex}{1.5em}
        \big(\il-\overline{\il}\big)\big|_{\tau=0}=\il_o-\overline{\il}_o,
        \end{array}
        \right.
\end{align}
which can be solved using the integrating factor $\ue^\tau$. We have 
\begin{align}
    \il-\overline{\il} = \ue^{-\tau}\big(\il_o-\overline{\il}_o\big) 
    + \int_0^\tau \ue^{\tau'-\tau}\big(S-\overline{S}\big)\,\ud\tau',
\end{align}
and thus
\begin{align}
    \il = \overline{\il}_o + \ue^{-\tau}\big(\il_o-\overline{\il}_o\big) 
    + \int_0^\tau \Big\{ \overline{S} + \ue^{\tau'-\tau}\big(S-\overline{S}\big)\Big\}\,\ud\tau'.
\end{align}
Let 
\begin{align} \label{theta-infinity}
    \il_{\infty}:=\overline{\il}_{\infty}=\lim_{\tau\rt\infty}\overline{\il}=\overline{\il}_o + \int_0^\infty \!\overline{S}\,\ud\tau.
\end{align}
Then it is straightforward to verify \eqref{initial-layer-decay'}.
\end{proof}

Let $\il_0$ be the solution to \eqref{Pb:initial-layer} with $S=0$ and the initial data $\il_{0,o}=u_o$ satisfying \eqref{assumption}. Then by \eqref{theta-infinity} we have $\il_{0,\infty}=\overline{u}_o$.
Define the leading-order initial layer to be
\begin{align}\label{initial-layer-0}
\uuu_0(\tau,\vx,\vw):=\il_0(\tau,\vx,\vw)-\il_{0,\infty}(x),
\end{align}
so that it satisfies
\begin{align}\label{Pb:initial-layer-0}
        \left\{
        \begin{array}{l}
        \p_{\tau}\uuu_0+\uuu_0-\overline{\uuu_0}=0,\\\rule{0ex}{1.5em}
        \uuu_0(0,\vx,\vw)=u_o(\vx,\vw)-\overline{u}_o(\vx),\\\rule{0ex}{1.5em}
        \ds\lim_{\tau\rt\infty}\uuu_0(\tau,\vx,\vw)=0,
        \end{array}
        \right.
\end{align}
as well as 
\begin{align}
    \nnm{\ue^{\tau}\uuu_0}_{L^{\infty}_\tau W^{3,\infty}_xL^{\infty}_w}\ls 1.
\end{align}

Next, we define the second-order initial layer to be
\begin{align}\label{initial-layer-1}
\uuu_1(\tau,\vx,\vw):=\il_1(\tau,\vx,\vw)-\il_{1,\infty}(x),
\end{align}
where $\il_1$ is determined by solving \eqref{Pb:initial-layer} with $S=-\vw\cdot\nx\uuu_0$ and $\il_{1,o}=\vw\cdot\nx\u_0(0,\vx,\vw)$, 
and so $\il_{1,\infty}(x)=\lim_{\tau\rt\infty}\overline{\il}_1(\tau,\vx)$ whose explicit formula is given by \eqref{theta-infinity}.
Then $\uuu_1$ satisfies 
\begin{align}\label{Pb:initial-layer-1}
        \left\{
        \begin{array}{l}
        \p_{\tau}\uuu_1+\uuu_1-\overline{\uuu_1}=-\vw\cdot\nx\uuu_0,\\\rule{0ex}{1.5em}
        \uuu_1(0,\vx,\vw)=\vw\cdot\nx\u_0(0,\vx,\vw)-\il_{1,\infty}(\vx),\\\rule{0ex}{1.5em}
        \ds\lim_{\tau\rt\infty}\uuu_1(\tau,\vx,\vw)=0,
        \end{array}
        \right.
\end{align}
as well as 
\begin{align}\label{theta_1,infty}
    \nm{\il_{1,\infty}}_{W^{2,\infty}_x}+
    \nnm{\ue^{\tau}\uuu_1}_{L^{\infty}_\tau W^{2,\infty}_xL^{\infty}_w}\ls 1.
\end{align}

%%%%%%%%%%%%%%%%%%%%%%%%%%%%%%%%%%%%%%%%%%%%%%%%%%%%%%%%%%%%%%%%%%%%%%%%
\subsection{Boundary Layer}\label{sec:boundary-layer}
%%%%%%%%%%%%%%%%%%%%%%%%%%%%%%%%%%%%%%%%%%%%%%%%%%%%%%%%%%%%%%%%%%%%%%%%

We first perform change of coordinates depending on the geometry of the boundary.
For the smooth surface $\p\Omega$, there exists an orthogonal curvilinear coordinates system $(\iota_1,\iota_2)$ such that the coordinate lines coincide with the principal directions at any $\vx_0\in\p\Omega$. Assume that $\p\Omega$ is parameterized by $\vr=\vr(\iota_1,\iota_2)$, and let $\vt_i:=\pl_i^{-1}\p_{\iota_i}\vr$ ($i=1,2$) be the two orthogonal unit tangential vectors, where $\pl_i:=\abs{\p_{\iota_i}\vr}$.
Then $\{\vt_1,\vt_2,n\}$ form an orthonormal basis at each point in a neighborhood of the boundary.
Hence, we can define a new coordinate system $(\mn, \iota_1,\iota_2)$ near the boundary, where $\mn$ denotes the normal distance to the boundary surface $\p\Omega$ so that $\vx=\vr-\mn\vn$.
We then define the rescaled normal spatial variable $\eta=\e^{-1}\mn$ again due to the principle of dominant balance, and so $\p_\mn=\e^{-1}\p_\eta$.

Define also the spherical substitution $w\mapsto(\phi,\psi)$ for the velocity $w\in\s^2$ via
\begin{align}\label{velocity}
-\vw\cdot\vn=\sin\phi,\quad
\vw\cdot\vt_1=\cos\phi\sin\psi,\quad
\vw\cdot\vt_2=\cos\phi\cos\psi,
\qquad \text{for }\: \phi\in[-\tfrac{\pi}{2},\tfrac{\pi}{2}],\;\, 
\psi\in[0,2\pi).
\end{align}

Under the change of variables $(\vx,\vw)\mapsto(\eta,\iota_1,\iota_2;\phi,\psi)$ in the phase space, the transport operator becomes
\begin{align}\label{expand 6}
\e\dt+\vw\cdot\nx=&\;\e\dt+\e^{-1}\sin\phi\,\frac{\p}{\p\eta}
+\frac{\rc_1\cos\phi\sin\psi}{\pl_1(\rc_1-\e\eta)}\frac{\p}{\p\iota_1}+\frac{\rc_2\cos\phi\cos\psi}{\pl_2(\rc_2-\e\eta)}\frac{\p}{\p\iota_2} \\
&\; -\bigg(\frac{\sin^2\psi}{\rc_1-\e\eta}+\frac{\cos^2\psi}{\rc_2-\e\eta}\bigg)\cos\phi\,\frac{\p}{\p\phi} \no\\
&\;+\frac{\sin\psi}{\rc_1-\e\eta}\bigg\{\frac{\rc_1\cos\phi}{\pl_1\pl_2}\Big[\vt_1\cdot\Big(\vt_2\times\big(\p_{\iota_1\iota_2}\vr\times\vt_2\big)\Big)\Big]
-\sin\phi\cos\psi\bigg\}\frac{\p}{\p\psi}\no\\
&\;
-\frac{\cos\psi}{\rc_2-\e\eta}\bigg\{\frac{\rc_2\cos\phi}{\pl_1\pl_2}\Big[\vt_2\cdot\Big(\vt_1\times\big(\p_{\iota_1\iota_2}\vr\times\vt_1\big)\Big)\Big]
-\sin\phi\sin\psi\bigg\}\frac{\p}{\p\psi},\no
\end{align}
where $R_i$ ($i=1,2$) is the radius of curvature.

Inserting the expansion $\uu\sim \sum_{k=0}^{\infty}\uu_k\,\e^k$ into the rescaled equation and comparing the order of $\e$ (see \eqref{expand 6}), we expect that the standard leading-order boundary layer $\uu_0(t;\eta,\iota_1,\iota_2;\phi,\psi)$ should satisfy
\begin{align}
\sin\phi\,\frac{\p\uu_0}{\p\eta}+\uu_0-\overline{\uu_0}=0.
\end{align}

Consider the general boundary layer problem (Milne problem) 
\begin{align} \label{Pb:boundary-layer}
\sin\phi\,\frac{\p\bl}{\p\eta}+\bl-\bbl=S,\qquad
\bbl(t;\eta,\iota_1,\iota_2):=\frac{1}{4\pi}\int_{-\pi}^{\pi}\int_{-\frac{\pi}{2}}^{\frac{\pi}{2}}\bl(t;\eta,\iota_1,\iota_2;\phi,\psi)\cos\phi\,\ud{\phi}\ud{\psi},
\end{align}
with boundary condition
\begin{align} \label{Pb:boundary-layer-BC}
\bl(t;0,\iota_1,\iota_2;\phi,\psi)=\rho(t;\iota_1,\iota_2;\phi,\psi)\ \ \text{ for}\ \ \sin\phi>0.
\end{align}
We are interested in the solution $\bl(t;\eta,\iota_1,\iota_2;\phi,\psi)$ that satisfies
\begin{align} \label{Pb:boundary-layer-decay}
    \lim_{\eta\rt\infty}\bl(t;\eta,\iota_1,\iota_2;\phi,\psi)=\bl_{\infty}(t;\iota_1,\iota_2)
\end{align}
for some $\bl_{\infty}(t;\iota_1,\iota_2)$ that does not depend on the velocity variables $(\phi,\psi)$. 

According to \cite[Section 4]{AA009}, we have the well-posedness and regularity of the above Milne problem:

\begin{proposition}[Milne problem] \label{prop:boundary-wellposedness}
    Assume 
    \begin{align}
        \nm{\rho}_{W^{2,\infty}_tW^{3,\infty}_{\iota_1,\iota_2}W^{1,\infty}_{w}} +\nnm{\ue^{\alpha\eta}S}_{W^{2,\infty}_tW^{3,\infty}_xW^{1,\infty}_w} \ls 1
    \end{align}
    for some $\alpha>0$.
    Then there exist a unique solution $\bl\in W^{2,\infty}_tW^{3,\infty}_{\iota_1,\iota_2}W^{1,\infty}_\psi L^{\infty}_{\eta,\phi}$ to \eqref{Pb:boundary-layer}\eqref{Pb:boundary-layer-BC} and a $\bl_{\infty}\in W^{2,\infty}_tW^{3,\infty}_{\iota_1,\iota_2}$ such that \eqref{Pb:boundary-layer-decay} holds in the sense of 
    \begin{align} \label{boundary-layer-decay}
    \nm{\bl_{\infty}}_{W^{2,\infty}_tW^{3,\infty}_{\iota_1,\iota_2}}
    +\nnm{\ue^{\beta\eta}\big(\bl-\bl_{\infty}\big)}_{W^{2,\infty}_tW^{3,\infty}_{\iota_1,\iota_2}W^{1,\infty}_\psi L^{\infty}_{\eta,\phi}}&\ls 1,\\
    \nnm{\ue^{\beta\eta}\sin\phi\,\p_\eta\big(\bl-\bl_{\infty}\big)}_{L^\infty}
    + \nnm{\ue^{\beta\eta}\sin\phi\,\p_\phi\big(\bl-\bl_{\infty}\big)}_{L^\infty}&\ls 1
    \end{align}
    for any $0<\beta< \min\{1,\alpha\}$.
\end{proposition}

Let $\bl_0$ be the solution to \eqref{Pb:boundary-layer} with $S=0$ and the boundary data $\rho(t;\iota_1,\iota_2;\phi,\psi)=g(t,\vx_0,\vw)$ satisfying \eqref{assumption}, and so there exists $\bl_{0,\infty}(t;\iota_1,\iota_2)=\lim_{\eta\rt\infty}\bl_0(t;\eta,\iota_1,\iota_2;\phi,\psi)$. 
Then $\bll_0:=\bl_0-\bl_{0,\infty}$ satisfies 
    \begin{align}
        \left\{ 
        \begin{array}{l}
        \sin\phi\,\dfrac{\p\bll_0}{\p\eta}+\bll_0-\bbll_0=0,\\\rule{0ex}{1.2em}
        \bll_0\big|_{\{\eta=0,\,\sin\!\phi>0\}}=g-\bl_{0,\infty},\\\rule{0ex}{1.5em}
        \ds\lim_{\eta\rt\infty}\bll_0=0,
        \end{array}
        \right.
    \end{align}
as well as
\begin{align} 
    \nm{\bl_{0,\infty}}_{W^{2,\infty}_tW^{3,\infty}_{\iota_1,\iota_2}}
    +\nnm{\ue^{\beta\eta}\bll_0}_{W^{2,\infty}_tW^{3,\infty}_{\iota_1,\iota_2}W^{1,\infty}_\psi L^{\infty}_{\eta,\phi}}&\ls 1,\label{Phi_0,infty}\\
    \nnm{\ue^{\beta\eta}\sin\phi\,\p_\eta\bll_0}_{L^\infty}
    + \nnm{\ue^{\beta\eta}\sin\phi\,\p_\phi\bll_0}_{L^\infty}&\ls 1
    \end{align}
for any $0<\beta< 1$.

Additionally, let $\chi(r)\in C^{\infty}(\r)$
and $\ch(r)=1-\chi(r)$ be smooth cutoff functions satisfying $\chi(r)=1$ for $\abs{r}\leq1$ and $\chi(r)=0$ for $\abs{r}\geq2$.
We define the cutoff boundary layer %\red{(cutoff/modified)}
\begin{align}\label{boundary layer}
\uu_0(t;\eta,\iota_1,\iota_2;\phi,\psi):=\chi(\e^{\frac{1}{2}}\eta)\ch(\e^{-1}\phi)\bll_0(t;\eta,\iota_1,\iota_2;\phi,\psi).
\end{align}
Here, the space cutoff $\chi(\e^{\frac{1}{2}}\eta)$ restricts the boundary layer effect within a thin layer near the boundary and helps avoid self-interactions, while the velocity cutoff $\ch(\e^{-1}\phi)$ truncates the grazing singularity.
With these cutoffs, the modified boundary layer $\uu_0$ satisfies 
\begin{align}\label{Pb:boundary-layer'}
        \left\{
        \begin{array}{l}
        \sin\phi\,\dfrac{\p\uu_0}{\p\eta}+\uu_0-\overline{\uu_0}=
        \sin\phi\cdot\e^{\frac{1}{2}}\chi'(\e^{\frac{1}{2}}\eta)\ch(\e^{-1}\phi)\bll_0 
        + \chi(\e^{\frac{1}{2}}\eta)\Big[\overline{\chi(\e^{-1}\phi)\bll_0} -\chi(\e^{-1}\phi)\overline{\bll}_0 \Big], \\\rule{0ex}{1.5em}
        \uu_0(t;0,\iota_1,\iota_2;\phi,\psi)=\ch(\e^{-1}\phi)\Big[g(t;\iota_1,\iota_2;\phi,\psi)-\bl_{0,\infty}(t;\iota_1,\iota_2)\Big] \ \ \text{ for}\ \ \sin\phi>0,\\\rule{0ex}{1.5em}
        \ds\lim_{\eta\rt\infty}\uu_0(t;\eta,\iota_1,\iota_2;\phi,\psi)=0.
        \end{array}
        \right.
\end{align}
Moreover, it holds that 
\begin{align} \label{boundary-layer-decay''}
    \nnm{\ue^{\beta\eta\,}\uu_0}_{W^{2,\infty}_tW^{3,\infty}_{\iota_1,\iota_2}W^{1,\infty}_\psi L^{\infty}_{\eta,\phi}}\ls 1
    \end{align}
for any $0<\beta< 1$.

%%%%%%%%%%%%%%%%%%%%%%%%%%%%%%%%%%%%%%%%%%%%%%%%%%%%%%%%%%%%%%%%%%%%%%%%%%%%%%%%%%
\subsection{Matching Procedure}\label{Subsec:matching}
%%%%%%%%%%%%%%%%%%%%%%%%%%%%%%%%%%%%%%%%%%%%%%%%%%%%%%%%%%%%%%%%%%%%%%%%%%%%%%%%%%

Now we construct the initial layer, boundary layer, and interior solution for each level of the asymptotic expansion in \eqref{expand} via a matching procedure.

%%%%%%%%%%%%%%%%%%%%%%%%%%%%%%%%%%%%%%%%%%%%%%%%%%%%%%%%%%%%%%%%%%%%%%%%%%%%%%%%%%
\paragraph{\underline{\textit{Construction of $\uuu_0$, $\uu_0$, and $\u_0$}}}
%%%%%%%%%%%%%%%%%%%%%%%%%%%%%%%%%%%%%%%%%%%%%%%%%%%%%%%%%%%%%%%%%%%%%%%%%%%%%%%%%%

For the leading-order expansion,
we want to enforce the matching initial condition 
\begin{align}
    \big(\u_0+\uuu_0\big)\big|_{t=0} = u_o ,
\end{align}
and the matching boundary condition 
\begin{align}
\big(\u_0+\uu_0\big)\big|_{\Gamma_-} = g +O(\e^{0+}) .
\end{align}
With the initial layer $\uuu_0$ given in \eqref{initial-layer-0} that satisfies \eqref{Pb:initial-layer-0}
and the boundary layer $\uu_0$ given in \eqref{boundary layer} that satisfies \eqref{Pb:boundary-layer'},
we require the interior solution $\u_0$ to satisfy the following initial-boundary value problem (combining \eqref{expand 4}):
\begin{align}\label{U_0-pb}
\left\{
\begin{array}{l}
\u_0=\bu_0,\quad
\dt\bu_0-\Delta_x\bu_0=0,\\\rule{0ex}{1.5em}
\overline\u_0(0,\vx)=\overline{u}_o(\vx),\\\rule{0ex}{1.5em}
\overline\u_0(t,\vx_0)=\bl_{0,\infty}(t;\iota_1,\iota_2)\ \ \text{ for}\ \
\vx_0\in\p\Omega,
\end{array}
\right.
\end{align}
which uniquely determines $\u_0(t,\vx,\vw)$ due to classical theory for the heat equation.
Also, with \eqref{assumption}, \eqref{Phi_0,infty}, and by the standard parabolic estimate (cf. \cite{Krylov2008}), we have
for any $2\leq\NN<\infty$ 
\begin{align}\label{expand 10}
    \nnm{\u_0}_{W^{2,\infty}_tW^{3,\NN}_x}+\nm{\u_0}_{W^{2,\infty}_tW^{3,\NN}_{\iota_1,\iota_2}}\ls 1.
\end{align}
Then the joint boundary value is 
\begin{align}
\big(\u_0+\uu_0\big)\big|_{\Gamma_-} = g - \chi(\e^{-1}\phi)\big(g-\bl_{0,\infty}\big) .
\end{align}
Here the additional term $\chi(\e^{-1}\phi)\big(g-\bl_{0,\infty}\big)\sim O(\e^{\frac{1}{p}})$ in $L^p_w$-norm due to the velocity cutoff.

%%%%%%%%%%%%%%%%%%%%%%%%%%%%%%%%%%%%%%%%%%%%%%%%%%%%%%%%%%%%%%%%%%%%%%%%%%%%%%%%%%
\paragraph{\underline{\textit{Construction of $\uuu_1$ and $\u_1$}}}
%%%%%%%%%%%%%%%%%%%%%%%%%%%%%%%%%%%%%%%%%%%%%%%%%%%%%%%%%%%%%%%%%%%%%%%%%%%%%%%%%%

For the next-order expansion,
we want to enforce the matching initial condition 
\begin{align}
    \big(\u_1+\uuu_1\big)\big|_{t=0} = 0 .
\end{align}
With the initial layer $\uuu_1$ given in \eqref{initial-layer-1} that satisfies \eqref{Pb:initial-layer-1},
we require the interior solution $\u_1$ to satisfy (combining \eqref{expand 4'}):
\begin{align}\label{U_1-pb}
\left\{
\begin{array}{l}
\u_1=\bu_1-\vw\cdot\nx\u_{0},\quad
\dt\bu_1-\Delta_x\bu_1=0,\\\rule{0ex}{1.5em}
\overline\u_1(0,\vx)=\il_{1,\infty}(\vx),\\\rule{0ex}{1.5em}
\overline\u_1(t,\vx_0)=0 \ \ \text{ for}\ \
\vx_0\in\p\Omega,
\end{array}
\right.
\end{align}
which uniquely determines $\u_1(t,\vx,\vw)$.
Also, with \eqref{expand 10} and \eqref{theta_1,infty}, we have
for any $2\leq\NN<\infty$
\begin{align}\label{expand 11}
    \nnm{\u_1}_{W^{2,\infty}_tW^{2,\NN}_xL^\infty_w}+\nm{\u_1}_{W^{2,\infty}_tW^{2,\NN}_{\iota_1,\iota_2}L^{\infty}_w}\ls 1.
\end{align}

%%%%%%%%%%%%%%%%%%%%%%%%%%%%%%%%%%%%%%%%%%%%%%%%%%%%%%%%%%%%%%%%%%%%%%%%%%%%%%%%%%
\paragraph{\underline{\textit{Construction of $\u_2$}}}
%%%%%%%%%%%%%%%%%%%%%%%%%%%%%%%%%%%%%%%%%%%%%%%%%%%%%%%%%%%%%%%%%%%%%%%%%%%%%%%%%%

Lastly, we solve for a higher-order expansion $\u_2(t,\vx,\vw)$ from (see \eqref{expand 4''})
\begin{align}\label{U_2-pb}
\left\{
\begin{array}{l}
\u_2=\bu_2-\vw\cdot\nx\u_{1}-\dt\u_0,\quad
\dt\bu_2-\Delta_x\bu_2=0,\\\rule{0ex}{1.5em}
\overline\u_2(0,\vx)=0,\\\rule{0ex}{1.5em}
\overline\u_2(t,\vx_0)=0 \ \ \text{ for}\ \
\vx_0\in\p\Omega.
\end{array}
\right.
\end{align}
In view of \eqref{expand 10} and \eqref{expand 11}, we have
for any $2\leq\NN<\infty$
\begin{align}\label{expand 11'}
    \nnm{\u_2}_{W^{1,\infty}_tW^{1,\NN}_xL^\infty_w}+\nm{\u_2}_{W^{1,\infty}_tW^{1,\NN}_{\iota_1,\iota_2}L^{\infty}_w}\ls 1.
\end{align}

To summarize, we have obtained the well-posedness and regularity estimates of the interior solution, initial layer, and boundary layer:
\begin{proposition}\label{prop:wellposedness}
    Assume \eqref{assumption} holds for the initial data $u_o$ and boundary data $g$.
    Let $\u_0, \u_1, \u_2$ be constructed via \eqref{U_0-pb}\eqref{U_1-pb}\eqref{U_2-pb}, $\uuu_0, \uuu_1$ in Subsection~\ref{sec:initial-layer}, and $\uu_0$ in Subsection~\ref{sec:boundary-layer}.
    Then we have for any $2\leq\NN<\infty$
    \begin{align}
        &\nnm{\u_0}_{W^{2,\infty}_tW^{3,\NN}_x}+\nm{\u_0}_{W^{2,\infty}_tW^{3,\NN}_{\iota_1,\iota_2}}\ls 1,\label{U_0-est}\\
        &\nnm{\u_1}_{W^{2,\infty}_tW^{2,\NN}_xL^\infty_w}+\nm{\u_1}_{W^{2,\infty}_tW^{2,\NN}_{\iota_1,\iota_2}L^{\infty}_w}\ls 1,\label{U_1-est}\\
        &\nnm{\u_2}_{W^{1,\infty}_tW^{1,\NN}_xL^\infty_w}+\nm{\u_2}_{W^{1,\infty}_tW^{1,\NN}_{\iota_1,\iota_2}L^{\infty}_w}\ls 1,\label{U_2-est}
    \end{align}
    and 
    \begin{align}
        &\nnm{\ue^{\tau}\uuu_0}_{L^{\infty}_\tau W^{3,\infty}_xL^{\infty}_w}\ls 1,\label{U^I_0-est}\\
        &\nnm{\ue^{\tau}\uuu_1}_{L^{\infty}_\tau W^{2,\infty}_xL^{\infty}_w}\ls 1,\label{U^I_1-est}\\
       &\nnm{\ue^{\beta\eta\,}\uu_0}_{W^{2,\infty}_tW^{3,\infty}_{\iota_1,\iota_2}W^{1,\infty}_\psi L^{\infty}_{\eta,\phi}}\ls 1 \label{U^B_0-est}
    \end{align}
    for any $0<\beta< 1$.
\end{proposition}

\bigskip
%%%%%%%%%%%%%%%%%%%%%%%%%%%%%%%%%%%%%%%%%%%%%%%%%%%%%%%%%%%%%%%%%%%%%%%%%%%%%%%%%%
\section{Remainder Equation}\label{Sec:remainder-eq}
%%%%%%%%%%%%%%%%%%%%%%%%%%%%%%%%%%%%%%%%%%%%%%%%%%%%%%%%%%%%%%%%%%%%%%%%%%%%%%%%%%

Denote the approximate solution 
\begin{align}
    \ua:=\left(\u_0+\e\u_1+\e^2\u_2\right)+\left(\uuu_0+\e\uuu_1\right)+\uu_0,
\end{align}
and so the remainder
\begin{align}\label{expand-R}
    \re:= u^\e - \ua.
\end{align}
Inserting \eqref{expand-R} into \eqref{transport}, we have
\begin{align}
    &\e\dt\big(\ua+\re\big)+\vw\cdot\nx\big(\ua+\re\big)+\e^{-1}\big(\ua+\re\big)-\e^{-1}\big(\bua+\bre\big)=0,\\
    &\big(\ua+\re\big)\big|_{t=0}=u_o,\qquad \big(\ua+\re\big)\big|_{\Gamma_-}=g,\no
\end{align}
which yields
\begin{align}
    &\e\dt\re+\vw\cdot\nx\re+\e^{-1}\big(\re-\bre\big)
    =-\e\dt\ua -\vw\cdot\nx\ua-\e^{-1}\big(\ua-\bua\big),\\ 
    &\re|_{t=0}=u_o-\ua|_{t=0},\qquad
    \re|_{\Gamma_-}=g-\ua|_{\Gamma_-}.\no
\end{align}

Consider the initial-boundary value problem for the remainder $\re(t,\vx,\vw)$:
\begin{align}\label{remainder}
\left\{
\begin{array}{l}\displaystyle
\e\dt\re+\vw\cdot\nabla_x \re+\e^{-1}\big(\ire\big)=\ss\ \ \text{ in}\ \ \rp\times\Omega\times\s^2,\\\rule{0ex}{1.5em}
\re(0,\vx,\vw)=\z(\vx,\vw)\ \ \text{ in}\ \ \Omega\times\s^2,\\\rule{0ex}{1.5em}
\re(t,\vx_0,\vw)=\g(t,\vx_0,\vw)\ \ \text{ for}\
\ \vx_0\in\p\Omega\ \ \text{and}\ \ \vw\cdot\vn(\vx_0)<0,
\end{array}
\right.
\end{align}
where 
\begin{align}
\bre(t,\vx)=\frac{1}{4\pi}\int_{\s^2}\re(t,\vx,\vw)\,\ud{\vw}.
\end{align}

In particular, we know the initial data is given by
\begin{align}\label{d:z}
    \z:=\e^2\vw\cdot\nx\u_1\big|_{t=0} + \e^2\dt\u_0\big|_{t=0},
\end{align}
and the boundary data is given by
\begin{align}\label{d:h}
    \g:=\chi(\e^{-1}\phi)\big(g-\bl_{0,\infty}\big) -\e\uuu_1\big|_{\Gamma_-}
    +\e\big(\vw\cdot\nx\u_0\big)\big|_{\Gamma_-} +\e^2\big(\vw\cdot\nx\u_1\big)\big|_{\Gamma_-}
    + \e^2\dt\bl_{0,\infty}.
\end{align}
Note that with the compatibility condition $u_o|_{\gamma}=g|_{t=0}=C(x_0)$ (see \eqref{compatibility}), we actually have $\uu_0\big|_{t=0}=0$ from the Milne problem and $\uuu_0\big|_{\Gamma_-}=0$ from the initial layer problem.

The source term can be split as
\begin{align}\label{d:ss}
    \ss&:= -\e\dt\ua -\vw\cdot\nx\ua-\e^{-1}\big(\ua-\bua\big)\\
    &\;=:\sss+\ssd+\ssa+\ssb+\ssc \no
\end{align}
with the expressions
\begin{align}
    \sss:=&\,-\e^2\vw\cdot\nx\u_2-\e^2\dt\u_1-\e^3\dt\u_2,\label{d:s0}\\
    \ssd:=&\,-\e\vw\cdot\nx\uuu_1,\label{d:s4}\\
    \ssa:=&\:\bigg(\dfrac{\sin^2\psi}{\rc_1-\e\eta}+\dfrac{\cos^2\psi}{\rc_2-\e\eta}\bigg)\cos\phi\,\dfrac{\p\uu_0}{\p\phi},\label{d:s1}\\
    \ssb:=&\;-\e\dt\uu_0 -\frac{\rc_1\cos\phi\sin\psi}{\pl_1(\rc_1-\e\eta)}\frac{\p\uu_0}{\p\iota_1} -\frac{\rc_2\cos\phi\cos\psi}{\pl_2(\rc_2-\e\eta)}\frac{\p\uu_0}{\p\iota_2}\label{d:s2}\\
    &\;-\frac{\sin\psi}{\rc_1-\e\eta}\bigg\{\frac{\rc_1\cos\phi}{\pl_1\pl_2}\Big[\vt_1\cdot\Big(\vt_2\times\big(\p_{\iota_1\iota_2}\vr\times\vt_2\big)\Big)\Big]
    -\sin\phi\cos\psi\bigg\}\frac{\p\uu_0}{\p\psi}\no\\
    &\;
    +\frac{\cos\psi}{\rc_2-\e\eta}\bigg\{\frac{\rc_2\cos\phi}{\pl_1\pl_2}\Big[\vt_2\cdot\Big(\vt_1\times\big(\p_{\iota_1\iota_2}\vr\times\vt_1\big)\Big)\Big]
    -\sin\phi\sin\psi\bigg\}\frac{\p\uu_0}{\p\psi},\no\\
    \ssc:=&\,-\e^{-1}\Big\{\sin\phi\cdot\e^{\frac{1}{2}}\chi'(\e^{\frac{1}{2}}\eta)\ch(\e^{-1}\phi)\bll_0 
        + \chi(\e^{\frac{1}{2}}\eta)\Big[\overline{\chi(\e^{-1}\phi)\bll_0} -\chi(\e^{-1}\phi)\overline{\bll}_0 \Big]\Big\}.\label{d:s3}
\end{align}

We now give preliminary estimates of the initial, boundary, and source terms specified above.
The proof is straightforward and largely based on Proposition~\ref{prop:wellposedness} and the certain definitions with rescaling and cutoffs (cf. \cite[Section~3.3]{AA020}).

\begin{lemma}\label{lem:source-est}
Assume \eqref{assumption}\eqref{compatibility} hold for $u_o$ and $g$.
For the initial term $\z(\vx,\vw)$ given in \eqref{d:z}, we have
\begin{align}\label{H-est}
    \tnm{\z} \ls \e^2. 
    \end{align}
For the boundary term $\g(t,\vx_0,\vw)$ given in \eqref{d:h}, we have
\begin{align}\label{G-est}
    \tnnms{\g}{\Gamma_-}\ls\e.
    \end{align}
For the source term $\ss(t,\vx,\vw)$ given in \eqref{d:ss}--\eqref{d:s3}, we have
\begin{align}
&\qquad \tnnm{\sss}\ls\e^2, \qquad \tnnm{\ssd} \ls\e^2,\label{S-is-il-est}\\
        &\tnnm{\big(1+\eta\big)\ssa}\ls1,\qquad
        \tnnm{\big(1+\eta\big)\ss^{B\!L}_2}\ls \e^{\frac{1}{2}},\label{S-bl-12-est}\\
        &\tnnm{\big(1+\eta\big)\ss^{B\!L}_3}\ls1,\qquad \nnm{\big(1+\eta\big)\ss^{B\!L}_3}_{L^2_tL^2_xL^1_w}\ls \e^{\frac{1}{2}}.\label{S-bl-3-est}
    \end{align}
Moreover, the boundary layer $\uu_0$ defined in \eqref{boundary layer} satisfies
\begin{align}\label{BL-est}
        \tnnm{\big(1+\eta\big)\uu_0} 
         \ls \e^{\frac{1}{2}}.
    \end{align}
\end{lemma}

In preparation for the a priori estimate of the remainder, we need to write the weak formulation of the remainder equation, which depends essentially on the validity of integration by parts for the transport operator (cf. \cite[Lemma 2.2]{Esposito.Guo.Kim.Marra2013}):

\begin{lemma}[Green's identity]\label{lem:green}
Assume $f(t,\vx,\vw),\ g(t,\vx,\vw)\in L^{\infty}\big([0,T]; L^2(\Omega\times\s^2)\big)$ with
$\dt f+\vw\cdot\nx f,\ \dt g+\vw\cdot\nx g\in L^2([0,T]\times\Omega\times\s^2)$ and $f,\
g\in L^2_{\Gamma}$. Then for almost all $t,s\in[0,T]$, we have
\begin{align}
&\int_s^t\iint_{\Omega\times\s^2}\Big\{\big(\dt f+\vw\cdot\nx f\big)g+\big(\dt g+\vw\cdot\nx
g\big)f\Big\}\\
=&\iint_{\Omega\times\s^2}f(t)g(t)-\iint_{\Omega\times\s^2}f(s)g(s)+\int_s^t\!\int_{\gamma}fg(\vw\cdot n).\no
\end{align}
\end{lemma}
 
With this result, the weak formulation of \eqref{remainder} takes the following form:
for any test function $\weak(t,x,w)\in L^{\infty}\big([0,T]; L^2(\Omega\times\s^2)\big)$ with $\dt\weak+\vw\cdot\nx\weak\in L^2([0,T]\times\Omega\times\s^2)$ and $\weak\in L^2_{\Gamma}$, 
it holds for any $t\in[0,T]$ that
\begin{align}\label{weak formulation}
    &\;\e\bbrb{\re(t),\weak(t)}{x,w}-\e\bbrb{\z,\weak(0)}{x,w}
    +\br{\re, \weak}_{\Gamma_+} - \br{\g, \weak}_{\Gamma_-} \\
    =&\; \e\dbbr{\re,\dt\weak}+\dbbr{\re, w\cdot\nx \weak}-\e^{-1}\dbbr{\re-\bre, \weak} +\dbr{\ss, \weak}.\no
\end{align}
Here, the time integral is taken over $[0,t]$.

\bigskip
%%%%%%%%%%%%%%%%%%%%%%%%%%%%%%%%%%%%%%%%%%%%%%%%%%%%%%%%%%%%%%%%%%%%%%%%%%%%%%%%%%
\section{Remainder Estimate}\label{Sec:remainder-est}
%%%%%%%%%%%%%%%%%%%%%%%%%%%%%%%%%%%%%%%%%%%%%%%%%%%%%%%%%%%%%%%%%%%%%%%%%%%%%%%%%%

Our goal of this section is to prove the remainder estimate for 
\begin{align}
    \re:= u^\e - \Big[\left(\u_0+\e\u_1+\e^2\u_2\right)+\left(\uuu_0+\e\uuu_1\right)+\uu_0\Big].
\end{align}

\begin{theorem}[Remainder estimate]\label{thm:remainder-est}
    Under the assumption \eqref{assumption} with \eqref{compatibility}, we have
    \begin{align}
    \nnm{\re}_{L^{\infty}_tL^2_{x,w}}+\e^{-\frac{1}{2}}\tnnms{\re}{\Gamma_+}+\e^{-\frac{1}{2}}\tnnm{\bre}+\e^{-1}\tnnm{\ire}\ls 1,
    \end{align}
where the implicit constant is independent of $\e$.
\end{theorem}

The combination of the following two lemmas, which will be presented in the rest of this section, will lead immediately to the desired estimate above.

\begin{lemma}[Energy estimate]\label{lem:energy}
    Under the assumption \eqref{assumption} with \eqref{compatibility}, we have
    \begin{align}\label{energy}
    \nnm{\re}_{L^{\infty}_tL^2_{x,w}}^2+\e^{-1}\tnnms{\re}{\Gamma_+}^2+\e^{-2}\tnnm{\ire}^2\ls \oo\e^{-1}\tnnm{\bre}^2+\oo^{-1}.
    \end{align}
Here $0<\oo\ll 1$ is a constant that can be taken sufficiently small.
\end{lemma}

\begin{proof}
Taking $\weak=\e^{-1}\re$ in the weak formulation \eqref{weak formulation} and using the fundamental theorem of calculus and divergence theorem, we obtain
\begin{align}
    \tfrac{1}{2}\tnm{\re(t)}^2-\tfrac{1}{2}\tnm{\z}^2+\tfrac{1}{2}\e^{-1}\tnnms{\re}{\Gamma_+}^2 - \tfrac{1}{2}\e^{-1}\tnnms{\g}{\Gamma_-}^2 
    =-\e^{-2}\dbr{\ire,\re} +\e^{-1}\dbr{\ss, \re}. 
\end{align}
Observing the orthogonality $\br{\bre,\ire}_{w}=0$, we then rearrange the terms to arrive at 
\begin{align}\label{Feimeng}
    \tfrac{1}{2}\tnm{\re(t)}^2+\tfrac{1}{2}\e^{-1}\tnnms{\re}{\Gamma_+}^2+\e^{-2}\tnnm{\ire}^2
    =\tfrac{1}{2}\tnm{\z}^2+\tfrac{1}{2}\e^{-1}\tnnms{\g}{\Gamma_-}^2 +\e^{-1}\dbr{\ss,\re}.
\end{align}

By \eqref{H-est} and \eqref{G-est}, we know 
\begin{align}\label{energy 01}
    \tnm{\z}^2+\e^{-1}\tnnms{\g}{\Gamma_-}^2 \ls \e.
\end{align}
Now we split the last term in \eqref{Feimeng} and estimate each part separately.
Using Young's inequality and by \eqref{S-is-il-est}, we have
\begin{align}\label{energy 02}
    \e^{-1}\abs{\dbbr{\sss+\ssd,\re}}\ls \oo\tnnm{\re}^2+\oo^{-1}\e^{-2}\tnnm{\sss+\ssd}^2\ls \oo\tnnm{\re}^2+\oo^{-1}\e^2.
\end{align}
Likewise, by \eqref{S-bl-12-est} and \eqref{S-bl-3-est}, we have 
\begin{align}\label{energy 03}
    \e^{-1}\abs{\dbbr{\ssa+\ssb+\ssc,\ire}}
    &\ls\oo\e^{-2}\tnnm{\ire}^2+\oo^{-1}\tnnm{\ssa+\ssb+\ssc}^2\\
    &\ls \oo\e^{-2}\tnnm{\ire}^2+\oo^{-1}.\no
\end{align} 
For $\e^{-1}\dbbr{\ssa,\bre}$, we use integration by parts in $\phi$ (note that the velocity integral contains a Jacobian $\cos\phi$) followed by \eqref{BL-est} and get
\begin{align}\label{energy 04}
    \e^{-1}\abs{\dbbr{\ssa,\bre }}&=\e^{-1}\abs{\dbr{ \Big(\tfrac{\sin^2\psi}{R_1-\e\eta}+\tfrac{\cos^2\psi}{R_2-\e\eta}\Big)\cos\phi\,\p_{\phi}\uu_0,\bre}}\\
    &=\e^{-1}\abs{\dbr{ \Big(\tfrac{\sin^2\psi}{R_1-\e\eta}+\tfrac{\cos^2\psi}{R_2-\e\eta}\Big)2\sin\phi\,\uu_0,\bre}}\no\\
    &\ls\e^{-1}\nnm{\uu_0}_{L^2_tL^2_xL^1_w}\tnnm{\bre}\ls\e^{-\frac{1}{2}}\tnnm{\bre}\ls \oo\e^{-1}\tnnm{\bre}^2+\oo^{-1}.\no
\end{align}
Turning to the remaining terms, we use \eqref{S-bl-12-est}\eqref{S-bl-3-est} to estimate
\begin{align}\label{energy 05}
    \e^{-1}\abs{\dbbr{\ssb+\ssc,\bre}}
    &\ls\e^{-1}\tnnm{\bre}\Big(\nnm{\ssb}_{L^2_tL^2_xL^1_w}+\nnm{\ssc}_{L^2_tL^2_xL^1_w}\Big)\\
    &\ls\e^{-\frac{1}{2}}\tnnm{\bre}\ls \oo\e^{-1}\tnnm{\bre}^2+\oo^{-1}.\no
\end{align}

Collecting \eqref{energy 02} through \eqref{energy 05}, we deduce that
\begin{align}\label{energy 06}
    \e^{-1}\abs{\dbr{\ss, \re}}\ls \oo\e^{-2}\tnnm{\ire}^2+\oo\e^{-1}\tnnm{\bre}^2+\oo^{-1}.
\end{align}
Finally, combining \eqref{Feimeng} with \eqref{energy 01} and \eqref{energy 06} settles \eqref{energy}.
\end{proof}

\begin{lemma}[Kernel estimate]\label{lem:kernel}
    Under the assumption \eqref{assumption} with \eqref{compatibility}, we have
    \begin{align}\label{kernel}
    \tnnm{\bre}^2\ls\e\nnm{\re}_{L^{\infty}_tL^2_{x,w}}^2+\tnnm{\ire}^2 +\tnnms{\re}{\Gamma_+}^2+\e.
    \end{align}
\end{lemma}

\begin{proof}
We will use an auxiliary Dirichlet problem for the Poisson equation to construct our test functions.
Let $\test(t,x):=(-\Delta_x)^{-1}\bre(t,x)$ be the solution to (for fixed $t$)
\begin{align}
\left\{
\begin{array}{l}
-\Delta_x\test=\bre\ \ \text{ in}\
\ \Omega,\\\rule{0ex}{1.2em}
\test=0\ \ \quad\text{ on}\ \
\p\Omega.
\end{array}
\right.
\end{align}
Based on the standard elliptic estimate and the trace theorem, for every $t\in\rp$ we have
\begin{align}\label{kernel 01}
    \nm{\test(t)}_{H^2(\Omega)}+\babs{\nx\test(t)}_{H^{\frac{1}{2}}(\p\Omega)}\ls\nm{\bre(t)}_{L^2(\Omega)}.
\end{align}
We first test \eqref{remainder} with $\weak=\test$ (or equivalently, take $\weak=\test$ in \eqref{weak formulation}) to get
\begin{align}
    \e\dbbr{\dt\re,\test} +\br{\re, \test}_{\Gamma_+} - \br{\g, \test}_{\Gamma_-}
    -\dbbr{\re,\vw\cdot\nx\test}+\e^{-1}\dbbr{\ire,\test}=\dbr{\ss,\test}.
\end{align}
Notice that $\br{\ire,\test}_{w}=0$ and $\br{\dt\re- \dt\bre,\test}_{w}=0$ by orthogonality, $\br{\bre,\vw\cdot\nx\test}_{w}=0$ by symmetry, and that the boundary terms vanish due to $\test|_{\p\Omega}=0$. Hence, we are left with 
\begin{align}\label{kernel 02}
    \e\dbbr{\dt\bre,\test}-\dbbr{\ire,\vw\cdot\nx\test}=\dbr{\ss,\test}.
\end{align}
We then test \eqref{remainder} with $\weak=w\cdot\nx\test$ to get 
\begin{align}\label{kernel 03}
    \e\dbbr{\dt\re,w\cdot\nx\test}+\bbr{\re, w\cdot\nx\test}_{\Gamma_+} - \bbr{\g, w\cdot\nx\test}_{\Gamma_-}\\
    -\dbbr{\re,\vw\cdot\nx\big(w\cdot\nx\test\big)}+\e^{-1}\dbbr{\ire,w\cdot\nx\test}
    &=\dbbr{\ss,w\cdot\nx\test}.\no
\end{align}
Adding $\e^{-1}\times$\eqref{kernel 02} and \eqref{kernel 03} to eliminate $\e^{-1}\dbbr{\ire,w\cdot\nx\test}$, we obtain
\begin{align}\label{kernel 04}
    \dbbr{\dt\bre,\test}+\e\dbbr{\dt\re,w\cdot\nx\test}+\bbr{\re, w\cdot\nx\test}_{\Gamma_+} - \bbr{\g, w\cdot\nx\test}_{\Gamma_-}\\
    -\dbbr{\re,\vw\cdot\nx\big(w\cdot\nx\test\big)}&=\e^{-1}\dbr{\ss,\test}+\dbbr{\ss,w\cdot\nx\test}.\no
\end{align}
On the one hand, we may split
\begin{align}
    -\dbbr{\re,\vw\cdot\nx\big(w\cdot\nx\test\big)}=
    -\dbbr{\bre,\vw\cdot\nx\big(w\cdot\nx\test\big)}
    -\dbbr{\ire,\vw\cdot\nx\big(w\cdot\nx\test\big)},
\end{align}
where the kernel part has the key positivity:
\begin{align}
    -\dbbr{\bre,\vw\cdot\nx\big(w\cdot\nx\test\big)}\simeq-\dbbr{\bre,\dx\test}=\tnnm{\bre}^2,
\end{align}
while the non-kernel part can be bounded as
\begin{align}
    \abs{\dbbr{\ire,\vw\cdot\nx\big(w\cdot\nx\test\big)}}\ls\tnnm{\ire}\nnm{\test}_{L^2_tH^2_x}\ls\oo\tnnm{\bre}^2+\oo^{-1}\tnnm{\ire}^2.
\end{align}
On the other hand, integration by parts reveals that
\begin{align}
    \dbbr{\dt\bre,\test}=-\dbbr{\dt\dx\test,\test}=\dbbr{\dt\nx\test,\nx\test}=\tfrac{1}{2}\btnm{\nx\test(t)}^2-\tfrac{1}{2}\btnm{\nx\test(0)}^2,
\end{align}
where \eqref{H-est} yields
\begin{align}
    \tfrac{1}{2}\btnm{\nx\test(0)}^2\ls\tnm{\bre(0)}^2=\tnm{\overline\z}^2 \ls \tnm{\z}^2\ls\e^4.
\end{align}
In addition, we bound
\begin{align}
    \e\dbbr{\dt\re,w\cdot\nx\test}&=\e\dbbr{\dt\left(\ire\right), w\cdot\nx\test}\\
    &=-\e\dbbr{\ire, w\cdot\nx\dt\test}
    +\e\bbrb{\left(\ire\right)(t),w\cdot\nx\test(t)}{x,w}-\e\bbrb{\left(\ire\right)(0),w\cdot\nx\test(0)}{x,w},\no
\end{align}
by
\begin{align}
\e\Babs{\dbbr{\ire, w\cdot\nx\dt\test}}&\ls\e\tnnm{\ire}\btnnm{\nx\dt\test}\ls\tnnm{\ire}^2+\e^2\btnnm{\nx\dt\test}^2,\\
    \e\abs{\bbrb{\left(\ire\right)(t),w\cdot\nx\test(t)}{x,w}}
    &\ls\e\tnm{\left(\ire\right)(t)}\nm{\test(t)}_{H^1}\ls\e\tnm{\re(t)}^2,\\
    \e\abs{\bbrb{\left(\ire\right)(0),w\cdot\nx\test(0)}{x,w}}
    &\ls \e\tnm{\re(0)}^2 =\e\tnm{\z}^2\ls\e^5.
\end{align}
Also, the boundary terms can be controlled using \eqref{kernel 01} and \eqref{G-est}:
\begin{align}\label{boundary-est}
    \abs{\bbr{\re, w\cdot\nx\test}_{\Gamma_+} - \bbr{\g, w\cdot\nx\test}_{\Gamma_-}}&\ls\Big(\tnnms{\re}{\Gamma_+}+\tnnms{\g}{\Gamma_-}\Big)\nm{\nx\test}_{L^2_{t,\p\Omega}}\\
    &\ls\oo\tnnm{\bre}^2+\oo^{-1}\tnnms{\re}{\Gamma_+}^2+\oo^{-1}\e^2.\no
\end{align}
Collecting \eqref{kernel 04} through \eqref{boundary-est}, we deduce that
\begin{align}\label{kernel 05}
    \btnm{\nx\test(t)}^2+\tnnm{\bre}^2\ls&\;\e^2\btnnm{\nx\dt\test}^2+\e\tnm{\re(t)}^2+\oo^{-1}\tnnm{\ire}^2 +\oo^{-1}\tnnms{\re}{\Gamma_+}^2+\oo^{-1}\e^2\\
    &\;+\e^{-1}\abs{\dbr{\ss,\test}}+\abs{\dbbr{\ss,w\cdot\nx\test}}.\no
\end{align}

Next, we estimate the term $\e^2\tnnm{\nx\dt\test}^2$ on the RHS above.
Let $\Test:=\dt\test=(-\Delta_x)^{-1}\dt\bre$ denote the solution to
\begin{align}
\left\{
\begin{array}{l}
-\Delta_x\Test=\dt\bre\ \ \text{ in}\
\ \Omega,\\\rule{0ex}{1.2em}
\Test=0\ \ \quad\text{ on}\ \
\p\Omega.
\end{array}
\right.
\end{align}
Poincar\'e's inequality indicates that
\begin{align}\label{kernel 08}
    \nm{\Test(t)}_{L^2(\Omega)}\ls\nm{\nx\Test(t)}_{L^2(\Omega)}=\bnm{\nx\dt\test(t)}_{L^2(\Omega)}.
\end{align}
Testing \eqref{remainder} against $\weak=\e\Test$ and similar to \eqref{kernel 02} we get
\begin{align}\label{kernel 06}
    \e^2\dbbr{\dt\bre,\Test}-\e\dbbr{\ire,\vw\cdot\nx\Test}=\e\dbr{\ss,\Test}.
\end{align}
Again, integration by parts reveals that
\begin{align}\label{kernel 06-1}
    \e^2\dbbr{\dt\bre,\Test}=
    -\e^2\dbbr{\dx\Test,\Test}=\e^2\tnnm{\nx\Test}^2=\e^2\btnnm{\nx\dt\test}^2.
\end{align}
Also, we have
\begin{align}\label{kernel 06-2}
    \e\abs{\dbbr{\ire,\vw\cdot\nx\Test}}\ls \e\tnnm{\ire}\tnnm{\nx\Test}\ls \oo\e^2\btnnm{\nx\dt\test}^2+\oo^{-1}\tnnm{\ire}^2.
\end{align}
Hence, \eqref{kernel 06} together with \eqref{kernel 06-1} and \eqref{kernel 06-2} yields
\begin{align}\label{kernel 07}
    \e^2\btnnm{\nx\dt\test}^2\ls \oo^{-1}\tnnm{\ire}^2+\e\abs{\dbr{\ss,\Test}}.
\end{align}
We then insert \eqref{kernel 07} into \eqref{kernel 05} and obtain
\begin{align}\label{kernel 09}
    \btnm{\nx\test(t)}^2+\e^2\btnnm{\nx\dt\test}^2+\tnnm{\bre}^2\ls&\;\e\tnm{\re(t)}^2+\oo^{-1}\tnnm{\ire}^2 +\oo^{-1}\tnnms{\re}{\Gamma_+}^2+\oo^{-1}\e^2\\
    &\;+\e^{-1}\abs{\dbr{\ss,\test}}+\abs{\dbbr{\ss,w\cdot\nx\test}}+\e\abs{\dbr{\ss,\Test}}.\no
\end{align}

Now it remains to handle the last three terms above involving the source term $\ss$.
By using \eqref{kernel 01}, \eqref{kernel 08}, and \eqref{S-is-il-est}, we see that
\begin{align}\label{kernel 10}
    &\;\e^{-1}\abs{\dbbr{\sss+\ssd, \test}}+\abs{\dbbr{\sss+\ssd, w\cdot\nx\test}}+\e\abs{\dbbr{\sss+\ssd,\Test}}\\
    \ls&\;\e^{-1}\Big(\tnnm{\sss}+\tnnm{\ssd}\Big)\nnm{\test}_{L^2_tH^1_x}+\e\Big(\tnnm{\sss}+\tnnm{\ssd}\Big)\tnnm{\Test}\no\\
    \ls&\;\e\tnnm{\bre}+\e^3\btnnm{\nx\dt\test}\ls \oo\tnnm{\bre}^2+\oo\e^2\btnnm{\nx\dt\test}^2+\oo^{-1}\e^2.\no
\end{align}
Analogous to \eqref{energy 04}, we may transfer $\ssa$ to $\uu_0$ via integration by parts in $\phi$.
Noting that $\test|_{\p\Omega}=0$, we then invoke Hardy's inequality and \eqref{S-bl-12-est}\eqref{S-bl-3-est}\eqref{BL-est} followed by \eqref{kernel 01} to bound
\begin{align}\label{kernel 11}
    &\,\e^{-1}\Babs{\dbbr{\ssa+\ssb+\ssc, \test}} \\
\ls&\;\e^{-1}\dBbr{\abs{\uu_0}+\abs{\ssb}+\abs{\ssc},\babs{\!\int_0^{\mu}\!\p_{\mu}\test}}
    =\dBbr{\abs{\eta\uu_0}+\abs{\eta\ssb}+\abs{\eta\ssc}, \Babs{\frac{1}{\mu}\int_0^{\mu}\!\p_{\mu}\test}}\no\\
    \ls&\, \Big(\bnnm{\eta\uu_0}_{L^2_tL^2_xL^1_w}+\bnnm{\eta\ssb}_{L^2_tL^2_xL^1_w}+\bnnm{\eta\ssc}_{L^2_tL^2_xL^1_w}\Big) 
    \tnnm{\frac{1}{\mu}\int_0^{\mu}\!\p_{\mu}\test}\no\\
    \ls&\;\e^{\frac{1}{2}}\tnnm{\p_{\mu}\test} \ls\e^{\frac{1}{2}}\nnm{\test}_{L^2_tH^1_x}
    \ls \e^{\frac{1}{2}}\tnnm{\bre}
    \ls\oo\tnnm{\bre}^2+\oo^{-1}\e.\no
\end{align}
In a parallel fashion, we have
\begin{align}\label{kernel 12}
    &\,\Babs{\dbbr{\ssa+\ssb+\ssc, w\cdot\nx\test}} \\
    \ls&\,\dBbr{\abs{\uu_0}+\abs{\ssb}+\abs{\ssc}, \Babs{\big(\nx\test\big)\big|_{\mu=0}+\int_0^{\mu}\!\p_\mu\big(\nx\test\big)}}\no\\
    \ls&\,\dBbr{\abs{\uu_0}+\abs{\ssb}+\abs{\ssc}, \abs{\nx\test}\big|_{\mu=0}}
    +\e\,\dBbr{\abs{\eta\uu_0}+\abs{\eta\ssb}+\abs{\eta\ssc}, \Babs{\frac{1}{\mu}\int_0^{\mu}\!\p_\mu\big(\nx\test\big)}}\no\\
    \ls&\;\e^{\frac{1}{2}}\nm{\nx\test}_{L^2_{t,\p\Omega}}
    +\e^{\frac{3}{2}}\tnnm{\p_\mu\big(\nx\test\big)}%\no\\
    \ls\e^{\frac{1}{2}}\nm{\nx\test}_{L^2_{t,\p\Omega}}+\e^{\frac{3}{2}}\nnm{\test}_{L^2_tH^2_x}\ls\e^{\frac{1}{2}}\tnnm{\bre}
    \ls\oo\tnnm{\bre}^2+\oo^{-1}\e,\no
\end{align}
and
\begin{align}\label{kernel 13}
    &\;\e\Babs{\dbbr{\ssa+\ssb+\ssc, \Test}} \\
    \ls&\;\e\,\dBbr{\abs{\uu_0}+\abs{\ssb}+\abs{\ssc},\babs{\!\int_0^{\mu}\!\p_\mu\Test}}
    =\e^2\dBbr{\abs{\eta\uu_0}+\abs{\eta\ssb}+\abs{\eta\ssc}, \babs{\frac{1}{\mu}\int_0^{\mu}\!\p_\mu\Test}}\no\\
    \ls&\;\e^2 \Big(\bnnm{\eta\uu_0}_{L^2_tL^2_xL^1_w}+\bnnm{\eta\ssb}_{L^2_tL^2_xL^1_w}+\bnnm{\eta\ssc}_{L^2_tL^2_xL^1_w}\Big) \tnnm{\frac{1}{\mu}\int_0^{\mu}\!\p_\mu\Test}\no\\
    \ls&\;\e^{\frac{5}{2}}\tnnm{\p_{\mu}\Test}
    \ls \e^{\frac{5}{2}}\btnnm{\nx\dt\test}
    \ls\oo\e^2\btnnm{\nx\dt\test}^2+\oo^{-1}\e^3.\no
\end{align}
Collecting \eqref{kernel 10} through \eqref{kernel 13}, we find
\begin{align}\label{kernel 14}
    \e^{-1}\abs{\dbr{\ss,\test}}+\abs{\dbbr{\ss,w\cdot\nx\test}}+\e\abs{\dbr{\ss,\Test}}
    \ls\oo\tnnm{\bre}^2+\oo\e^2\btnnm{\nx\dt\test}^2+\oo^{-1}\e.
\end{align}
Combined with \eqref{kernel 09}, this yields
\begin{align}
    \btnm{\nx\test(t)}^2+\e^2\btnnm{\nx\dt\test}^2+\tnnm{\bre}^2\ls\e\tnm{\re(t)}^2+\oo^{-1}\tnnm{\ire}^2 +\oo^{-1}\tnnms{\re}{\Gamma_+}^2 +\oo^{-1}\e,
\end{align}
and thus \eqref{kernel} follows.
\end{proof}

\bigskip
%%%%%%%%%%%%%%%%%%%%%%%%%%%%%%%%%%%%%%%%%%%%%%%%%%%%%%%%%%%%%%%%%%%%%%%%%%%%%%%%%%
\section{Diffusive Limit}\label{Sec:diffusive-limit}
%%%%%%%%%%%%%%%%%%%%%%%%%%%%%%%%%%%%%%%%%%%%%%%%%%%%%%%%%%%%%%%%%%%%%%%%%%%%%%%%%%

We are now ready to prove our main result for the in-flow boundary problem:

\begin{proof}[Proof of Theorem~\ref{main theorem}] 
The well-posedness of \eqref{transport} is rather standard and was addressed already in \cite{Bensoussan.Lions.Papanicolaou1979, Bardos.Santos.Sentis1984, AA016}.
The existence of $\u_0$ is guaranteed in Section~\ref{Subsec:matching}.
This leaves us to demonstrate the approximation estimate \eqref{main}.

Given the interior solutions $\u_0,\u_1,\u_2$, the initial layers $\uuu_0, \uuu_1$, and the boundary layer $\uu_0$ constructed in Section~\ref{Sec:asymptotic}, the remainder estimate (Theorem~\ref{thm:remainder-est}) implies
\begin{align}\label{main 01}
    \tnnm{u^{\e}-\u_0-\e\u_1-\e^2\u_2-\uuu_0-\e\uuu_1-\uu_0}\ls\e^{\frac{1}{2}}.
\end{align}    
On the other hand,
using the estimates recalled in Proposition~\ref{prop:wellposedness}, we see that
\begin{align}\label{main 02}
    \tnnm{\e\u_1+\e^2\u_2}\ls \e,\qquad
    \tnnm{\uuu_0+\e\uuu_1}\ls\e,\qquad
    \tnnm{\uu_0}\ls\e^{\frac{1}{2}},
\end{align}
where we take into account the temporal rescaling $\tau=\e^{-2}t$ for the initial layer and the normal spatial rescaling $\eta=\e^{-1}\mn$ for the boundary layer.
Consequently, \eqref{main} follows by combining \eqref{main 01} with \eqref{main 02}.
This completes the proof of the theorem.
\end{proof}

\bigskip
%%%%%%%%%%%%%%%%%%%%%%%%%%%%%%%%%%%%%%%%%%%%%%%%%%%%%%%%%%%%%%%%%%%%%%%%%%%%%%%%%%
\section{Diffuse-Reflection Boundary Problem}\label{Sec:diffuse-BC}
%%%%%%%%%%%%%%%%%%%%%%%%%%%%%%%%%%%%%%%%%%%%%%%%%%%%%%%%%%%%%%%%%%%%%%%%%%%%%%%%%%

Consider the unsteady neutron transport equation with the diffuse-reflection boundary condition:
\begin{align}\label{transport=}
\left\{
\begin{array}{l}\displaystyle
\e\dt u^{\e}+\vw\cdot\nabla_x u^{\e}+\e^{-1}\big(u^{\e}-\overline{u}^{\e}\big)=0\ \ \text{ in}\ \ \rp\times\Omega\times\s^2,\\\rule{0ex}{1.5em}
u^{\e}(0,\vx,\vw)=u_o(\vx,\vw)\ \ \text{ in}\ \  \Omega\times\s^2,\\\rule{0ex}{1.5em}
u^{\e}(t,\vx_0,\vw)=\pp[u^{\e}](t,\vx_0)+\e h(t,\vx_0,\vw)\ \ \text{ for}\
\ \vx_0\in\p\Omega\ \ \text{and}\ \ \vw\cdot\vn(\vx_0)<0.
\end{array}
\right.
\end{align}
Here, $\pp[u^{\e}]$ represents the weighted average over the outgoing boundary, i.e., 
\begin{align}
    \pp[u^{\e}](t,\vx_0):=c_{\gamma}\int_{\vw\cdot\vn>0}u^{\e}(t,\vx_0,\vw)(\vw\cdot\vn)\,\ud\vw,
\end{align}
where the constant $c_{\gamma}=\frac{1}{\pi}$ satisfies the normalization condition
$c_{\gamma}\int_{\vw\cdot\vn>0}(\vw\cdot\vn)\,\ud\vw=1$.
In addition, for the perturbative boundary data $h$, we assume the compatibility condition \eqref{g-compatibility=} 
to make sure $u^{\e}$ has zero flux \eqref{ue-zero-flux} at the boundary.

%%%%%%%%%%%%%%%%%%%%%%%%%%%%%%%%%%%%%%%%%%%%%%%%%%%%%%%%%%%%%%%%%%%%%%%%%%%%%%%%%%
\subsection{Asymptotic Analysis} 
%%%%%%%%%%%%%%%%%%%%%%%%%%%%%%%%%%%%%%%%%%%%%%%%%%%%%%%%%%%%%%%%%%%%%%%%%%%%%%%%%%

By analogy with what has gone in Section~\ref{Sec:asymptotic},
we expand the exact solution as
\begin{align}\label{expand=}
u^{\e}=\u+\uuu+\re=\left(\u_0+\e\u_1+\e^2\u_2\right)+\left(\uuu_0+\e\uuu_1\right)+\re.
\end{align}
Here, the initial layer $\uuu:=\uuu_0+\e\uuu_1$ and the corresponding $\il_{0,\infty}$, $\il_{1,\infty}$ can be taken the same as the in-flow case.
To determine the interior solutions, as noted in \cite{AA007}, we should instead enforce the Neumann boundary condition here. 
Specifically, $\u_0$ is solved from 
\begin{align}
\left\{
\begin{array}{l}
\u_0=\bu_0,\quad
\dt\bu_0-\Delta_x\bu_0=0,\\\rule{0ex}{1.5em}
\overline\u_0(0,\vx)=\overline{u}_o(\vx),\\\rule{0ex}{1.5em}
\tfrac{\p\overline\u_0}{\p\vn}(t,\vx_0)=0\ \ \text{ for}\ \
\vx_0\in\p\Omega.
\end{array}
\right.
\end{align}
Then $\u_1$ is solved from
\begin{align}
\left\{
\begin{array}{l}
\u_1=\bu_1-\vw\cdot\nx\u_{0},\quad
\dt\bu_1-\Delta_x\bu_1=0,\\\rule{0ex}{1.5em}
\overline\u_1(0,\vx)=\il_{1,\infty}(\vx),\\\rule{0ex}{1.5em}
\tfrac{\p\overline\u_1}{\p\vn}(t,\vx_0)=0 \ \ \text{ for}\ \
\vx_0\in\p\Omega,
\end{array}
\right.
\end{align}
and $\u_2$ from
\begin{align}
\left\{
\begin{array}{l}
\u_2=\bu_2-\vw\cdot\nx\u_{1}-\dt\u_0,\quad
\dt\bu_2-\Delta_x\bu_2=0,\\\rule{0ex}{1.5em}
\overline\u_2(0,\vx)=0,\\\rule{0ex}{1.5em}
\tfrac{\p\overline\u_2}{\p\vn}(t,\vx_0)=0 \ \ \text{ for}\ \
\vx_0\in\p\Omega.
\end{array}
\right.
\end{align}
Note that we do not need to introduce boundary layers here because the leading-order boundary condition is automatically satisfied by the interior solution $\u_0$ and the next-order boundary trace can be controlled in the remainder estimate. 

Moreover, under the assumption \eqref{assumption-2}, the well-posedness and corresponding regularity estimates as \eqref{U_0-est}--\eqref{U^I_1-est} in Proposition~\ref{prop:wellposedness} still hold.

%%%%%%%%%%%%%%%%%%%%%%%%%%%%%%%%%%%%%%%%%%%%%%%%%%%%%%%%%%%%%%%%%%%%%%%%%%%%%%%%%%
\subsection{Remainder Estimate}
%%%%%%%%%%%%%%%%%%%%%%%%%%%%%%%%%%%%%%%%%%%%%%%%%%%%%%%%%%%%%%%%%%%%%%%%%%%%%%%%%%

With a view to achieving the same approximation estimate \eqref{main-2}, we look at the remainder
\begin{align}
    \re:= u^\e - \Big[\left(\u_0+\e\u_1+\e^2\u_2\right)+\left(\uuu_0+\e\uuu_1\right)\Big],
\end{align}
and derive the initial-boundary value problem
\begin{align}\label{remainder=}
\left\{
\begin{array}{l}\displaystyle
\e\dt\re+\vw\cdot\nabla_x \re+\e^{-1}\big(\ire\big)=\ss\ \ \text{ in}\ \ \rp\times\Omega\times\s^2,\\\rule{0ex}{1.5em}
\re(0,\vx,\vw)=\z(\vx,\vw)\ \ \text{ in}\ \ \Omega\times\s^2,\\\rule{0ex}{1.5em}
\re(t,\vx_0,\vw)=\pp[\re](t,\vx_0)+\h(t,\vx_0,\vw)\ \ \text{ for}\
\ \vx_0\in\p\Omega\ \ \text{and}\ \ \vw\cdot\vn(\vx_0)<0,
\end{array}
\right.
\end{align}
where
\begin{align}
    \ss&:= -\e\vw\cdot\nx\uuu_1 -\e^2\vw\cdot\nx\u_2-\e^2\dt\u_1-\e^3\dt\u_2,\\[3pt]
    \z&:=\e^2\vw\cdot\nx\u_1\big|_{t=0} + \e^2\dt\u_0\big|_{t=0},\\[3pt]
    \h&:=\e h -\e\big(1-\pp\big)[\uuu_1]\big|_{\Gamma_-}
    +\e\big(1-\pp\big)[\vw\cdot\nx\u_0]\big|_{\Gamma_-} +\e^2\big(1-\pp\big)[\vw\cdot\nx\u_1]\big|_{\Gamma_-},
\end{align}
with the parallel estimates as in Lemma~\ref{lem:source-est}: 
\begin{align}
    \tnm{\ss} \ls \e^2, \qquad
    \tnm{\z} \ls \e^2, \qquad
    \tnnms{\h}{\Gamma_-}\ls\e.
\end{align}
Also, it can be verified that
\begin{align}\label{temp 1}
    \iint_{\Omega\times\s^2}\z(\vx,\vw) \,\ud\vw\ud\vx=0, \qquad
    \iint_{\Omega\times\s^2}\ss(t,\vx,\vw) \,\ud\vw\ud\vx=0 \ \ \text{ for all}\ \  t\in\rp,
\end{align}
\begin{align} \label{G-compatibility}
    \int_{\vw\cdot\vn<0}\h(t,\vx_0,\vw)(\vw\cdot\vn)\,\ud\vw=0\ \ \text{ for all}\ \  \vx_0\in\p\Omega\ \ \text{and}\ \ t\in\rp,
\end{align}
which yields
\begin{align}\label{R-zero-average}
    \int_{\Omega}\bre(t,x)\,\ud x=0\ \ \text{ for all}\ \  t\in\rp
\end{align}
and 
\begin{align}\label{R-zero-flux}
    \int_{\s^2}\re(t,\vx_0,\vw)(\vw\cdot\vn)\,\ud\vw = 0 
    \ \ \text{ for all}\ \  \vx_0\in\p\Omega\ \ \text{and}\ \ t\in\rp.
\end{align}

Observe that by the orthogonality $\bbr{\pp[\re],\big(1-\pp\big)[\re]}_{\gamma_+}=0$ and the compatibility \eqref{G-compatibility}, the boundary integral equals
\begin{align} \label{pangfufu}
    \int_{\Gamma}\re^2(\vw\cdot\vn)=\tnnms{\re}{\Gamma_+}^2-\btnnms{\pp[\re]+\h}{\Gamma_-}^2=\btnnms{\big(1-\pp\big)[\re]}{\Gamma_+}^2-\tnnms{\h}{\Gamma_-}^2.
\end{align}
Correspondingly,
we have the energy bound
\begin{align}\label{energy=}
    \nnm{\re}_{L^{\infty}_tL^2_{x,w}}^2+\e^{-1}\btnnms{\big(1-\pp\big)[\re]}{\Gamma_+}^2+\e^{-2}\tnnm{\ire}^2\ls \oo\e^{-1}\tnnm{\bre}^2+\oo^{-1}
\end{align}
with a small constant $0<\oo\ll 1$.

Paralleling Lemma~\ref{lem:kernel}, we may deduce the kernel estimate
\begin{align}\label{kernel=}
    \tnnm{\bre}^2\ls\e\nnm{\re}_{L^{\infty}_tL^2_{x,w}}^2+\tnnm{\ire}^2 +\btnnms{\big(1-\pp\big)[\re]}{\Gamma_+}^2+\e,
    \end{align}
by using test functions built from $\test(t,x):=(-\Delta_x)^{-1}\bre(t,x)$ with the Neumann boundary
condition instead.
To be specific, \eqref{R-zero-average} allows us to construct $\test$ via 
\begin{align}
\left\{
\begin{array}{l}
-\Delta_x\test=\bre\ \ \text{ in}\
\ \Omega,\\\rule{0ex}{1.2em}
\tfrac{\p\test}{\p\vn}=0\ \ \text{ on}\ \
\p\Omega,\\\rule{0ex}{1.2em}
\int_{\Omega}\test=0.
\end{array}
\right.
\end{align}
Consequently, we see the boundary terms
\begin{align}
    \int_{\Gamma}\re\test(\vw\cdot\vn) =
    \br{\re, \test}_{\Gamma_+} - \bbr{\pp[\re]+\h, \test}_{\Gamma_-}=0
\end{align}
from \eqref{R-zero-flux},
and 
\begin{align}
\int_{\Gamma}\re \big(w\cdot\nx\test\big)(\vw\cdot\vn) &=
    \bbr{\re, w\cdot\nx\test}_{\Gamma_+} - \bbr{\pp[\re]+\h, w\cdot\nx\test}_{\Gamma_-}\\
    &=\bbr{\big(1-\pp\big)[\re], w\cdot\nx\test}_{\Gamma_+} - \bbr{\h, w\cdot\nx\test}_{\Gamma_-},\no
\end{align}
thanks to $\tfrac{\p\test}{\p\vn}\big|_{\p\Omega} =(n\cdot\nx\test)\big|_{\p\Omega}=0$.
The remaining proof of \eqref{kernel=} then follows from analogous arguments.

Combining \eqref{energy=} with \eqref{kernel=} yields the remainder estimate
\begin{align}
    \nnm{\re}_{L^{\infty}_tL^2_{x,w}}+\e^{-\frac{1}{2}}\btnnms{\big(1-\pp\big)[\re]}{\Gamma_+}+\e^{-\frac{1}{2}}\tnnm{\bre}+\e^{-1}\tnnm{\ire}\ls 1,
\end{align}
and Theorem~\ref{main theorem-2} can thus be proved accordingly.

\bigskip
%%%%%%%%%%%%%%%%%%%%%%%%%%%%%%%%%%%%%%%%%%%%%%%%%%%%%%%%%%%%%%%%%%%%%%%%%%%%%%%%%%
\section{Specular-Reflection Boundary Problem}\label{Sec:specular-BC}
%%%%%%%%%%%%%%%%%%%%%%%%%%%%%%%%%%%%%%%%%%%%%%%%%%%%%%%%%%%%%%%%%%%%%%%%%%%%%%%%%%

Consider the unsteady neutron transport equation with the specular-reflection boundary condition:
\begin{align}\label{transport==}
\left\{
\begin{array}{l}\displaystyle
\e\dt u^{\e}+\vw\cdot\nabla_x u^{\e}+\e^{-1}\big(u^{\e}-\overline{u}^{\e}\big)=0\ \ \text{ in}\ \ \rp\times\Omega\times\s^2,\\\rule{0ex}{1.5em}
u^{\e}(0,\vx,\vw)=u_o(\vx,\vw)\ \ \text{ in}\ \  \Omega\times\s^2,\\\rule{0ex}{1.5em}
u^{\e}(t,\vx_0,\vw)=\sr[u^{\e}](t,\vx_0,\vw)+\e h(t,\vx_0,\vw)\ \ \text{ for}\
\ \vx_0\in\p\Omega\ \ \text{and}\ \ \vw\cdot\vn(\vx_0)<0,
\end{array}
\right.
\end{align}
where the specular-reflection operator 
\begin{align}
\sr[u^{\e}](t,\vx_0,\vw):=u^{\e}(t,\vx_0,\rr\vw)
 \quad\text{ with}\quad  \rr\vw:=\vw-2(\vw\cdot\vn)\vn,
\end{align} 
and $h$ is a given perturbation satisfying the compatibility condition \eqref{g-compatibility=},
so that $u^{\e}$ has zero flux at each boundary point.

%%%%%%%%%%%%%%%%%%%%%%%%%%%%%%%%%%%%%%%%%%%%%%%%%%%%%%%%%%%%%%%%%%%%%%%%%%%%%%%%%%
\subsection{Asymptotic Analysis} 
%%%%%%%%%%%%%%%%%%%%%%%%%%%%%%%%%%%%%%%%%%%%%%%%%%%%%%%%%%%%%%%%%%%%%%%%%%%%%%%%%%

Let us expand the exact solution as
\begin{align}\label{expand==}
u^{\e}=\u+\uuu+\uu+\re=\left(\u_0+\e\u_1+\e^2\u_2\right)+\left(\uuu_0+\e\uuu_1\right)+\left(\uu_0+\e\uu_1\right)+\re.
\end{align}
Here, the initial layer $\uuu:=\uuu_0+\e\uuu_1$ and the interior solution $\u:=\u_0+\e\u_1+\e^2\u_2$ can be constructed identically to the diffuse-reflection case.

However, unlike the previous scenario, some kind of ``boundary layer'' need to be introduced (if $h\neq 0$) in order to guarantee the perfect specular boundary condition for the remainder.
Specifically, we let $\uu_0=0$ because the leading-order interior solution $\u_0$ already satisfies the specular boundary condition.
We then define $\uu_1$ to be an extension of $h$ into a thin layer of the interior of domain:
\begin{align}
    \uu_1=
    \ge(t;\eta,\iota_1,\iota_2;\phi,\psi):= \chi(\eta) \mathds{1}_{\{\sin\phi>0\}} h(t;\iota_1,\iota_2;\phi,\psi),
\end{align}
and let it act as the next-order boundary layer.

%%%%%%%%%%%%%%%%%%%%%%%%%%%%%%%%%%%%%%%%%%%%%%%%%%%%%%%%%%%%%%%%%%%%%%%%%%%%%%%%%%
\subsection{Remainder Estimate}
%%%%%%%%%%%%%%%%%%%%%%%%%%%%%%%%%%%%%%%%%%%%%%%%%%%%%%%%%%%%%%%%%%%%%%%%%%%%%%%%%%

Proceeding in a parallel fashion,
we look at the remainder
\begin{align}
    \re:= u^\e - \Big[\left(\u_0+\e\u_1+\e^2\u_2\right)+\left(\uuu_0+\e\uuu_1\right)+\e\uu_1\Big],
\end{align}
and derive the initial-boundary value problem
\begin{align}\label{remainder==}
\left\{
\begin{array}{l}\displaystyle
\e\dt\re+\vw\cdot\nabla_x \re+\e^{-1}\big(\ire\big)=\ss\ \ \text{ in}\ \ \rp\times\Omega\times\s^2,\\\rule{0ex}{1.5em}
\re(0,\vx,\vw)=\z(\vx,\vw)\ \ \text{ in}\ \ \Omega\times\s^2,\\\rule{0ex}{1.5em}
\re(t,\vx_0,\vw)=\re(t,\vx_0,\rr\vw)\ \ \text{ for}\
\ \vx_0\in\p\Omega\ \ \text{and}\ \ \vw\cdot\vn(\vx_0)<0.
\end{array}
\right.
\end{align}
In particular, $\re$ satisfies the perfect specular boundary condition without any perturbation,
due to the introduction of the ``boundary layer'' $\uu_1$ 
and the observation that
\begin{align}
    \u_0=\sr[\u_0],\qquad\u_1=\sr[\u_1],\qquad \u_2=\sr[\u_2],
    \qquad \uuu_0=\sr[\uuu_0],\qquad \uuu_1=\sr[\uuu_1] 
\end{align}
from the construction of $\u_0,\u_1,\u_2,\uuu_0,\uuu_1$ and the compatibility conditions \eqref{compatibility-3}.

Compared to the diffuse-reflection case, the expression of $\z$ remains the same, while the source term $\ss$ now contains additional terms coming from $\uu_1$
(see \eqref{expand 6} for the transport operator under the boundary coordinates):
\begin{align}
\ss^{B\!L}:=
    -\e^2\dt \uu_1-\e\vw\cdot\nx \uu_1-\big(\uu_1-\overline{\uu_1}\big)
    =: \ss^{B\!L}_a + \ss^{B\!L}_b .
\end{align}
Here, let $\ss^{B\!L}_a:=\e\big(\ssa+\ssb\big)[\uu_1]$ denote the terms corresponding to $\ssa+\ssb$ in the inflow case, which now raise by one order of $\e$ and so $\tnnm{\big(1+\eta\big)\ss^{B\!L}_a}\ls \e^{\frac{3}{2}}$. 
In particular, the term containing $\p_\phi \uu_1$ is still good due to the continuity of $\ge$ across the grazing set where $\vw\cdot\vn=0$ under assumption \eqref{assumption-3} with the compatibility condition $h|_{\Gamma_0}=0$. 
The remaining part 
\begin{align}
\ss^{B\!L}_b:=
    -\sin\phi\, \p_{\eta}\uu_1 - \big(\uu_1-\overline{\uu_1}\big)
    = -\chi'(\eta) \sin\phi\,  \mathds{1}_{\{\sin\phi>0\}} h
    - \chi(\eta)\Big(\mathds{1}_{\{\sin\phi>0\}} h - \overline{\mathds{1}_{\{\sin\phi>0\}} h}\Big) 
\end{align} 
which corresponds to $\ssc$ in the inflow case
now has no cancellation to exploit since it does not satisfy the Milne problem. 
In view of the scaling $\eta=\e^{-1}\mu$, we see $\bnnm{\uu_1}_{L^2}\ls\e^{\frac{1}{2}}$
and those terms in $\ss^{B\!L}_b$ contribute to 
\begin{align}
    \tnnm{\big(1+\eta\big)\ss^{B\!L}}\ls \e^{\frac{1}{2}},
\end{align}
which is acceptable.

Besides, due to the presence of $\uu_1$, we have 
\begin{align}
    \iint_{\Omega\times\s^2}\ss = -\e^2  \iint_{\Omega\times\s^2} \dt \uu_1, 
    \quad\text{ and so }\quad
    \int_{\Omega}\bre = -\e \iint_{\Omega\times\s^2} \uu_1,
\end{align}
considering the null flux of $\re$ and $\uu_1$ under \eqref{g-compatibility=}.
Despite the fact that $\bre$ does not necessarily have zero average, we may redefine 
\begin{align}
    \widetilde\re:=\re-\frac{1}{|\Omega|}\int_{\Omega}\bre,\qquad
    \widetilde\ss:=\ss-\frac{1}{4\pi|\Omega|}\iint_{\Omega\times\s^2}\ss,
\end{align}
so that $\int_{\Omega}\overline{\widetilde\re}=0$, $\iint_{\Omega\times\s^2}\widetilde\ss=0$, and that
$\widetilde\re$ solves \eqref{remainder==} with $\ss$ replaced by $\widetilde\ss$.

Noticing that the difference $\abs{\int_{\Omega}\bre}\ls\e^2$ has higher order of smallness, we only need to obtain the remainder estimate for $\widetilde\re$. For simplicity, we will abuse notation $\re$ for $\widetilde\re$ in what follows.

The specular boundary condition on $\re$ yields
\begin{align}\label{R-zero-flux=}
    \int_{\s^2}\re(t,\vx_0,\vw)(\vw\cdot\vn)\,\ud\vw = 0 
    \ \ \text{ for all}\ \  \vx_0\in\p\Omega\ \ \text{and}\ \ t\in\rp,
\end{align}
as well as
\begin{align}\label{temp 3}
    \int_{\Gamma}\re^2(\vw\cdot\vn)=\tnnms{\re}{\Gamma_+}^2-\tnnms{\re}{\Gamma_-}^2=0.
\end{align}
Also, with the same test functions as in the diffuse-reflection case where $\tfrac{\p\test}{\p\vn}=0$ on $\p\Omega$, the corresponding boundary terms will vanish: 
\begin{align}
    \int_{\Gamma}\re\test(\vw\cdot\vn) &= \br{\re, \test}_{\Gamma_+} - \br{\re, \test}_{\Gamma_-}=0, \\
    \int_{\Gamma}\re \big(w\cdot\nx\test\big)(\vw\cdot\vn) &= \bbr{\re, w\cdot\nx\test}_{\Gamma_+} - \bbr{\re, w\cdot\nx\test}_{\Gamma_-}=0.\label{boundary-term}
\end{align}
Arguing analogously, we thus find
the energy bound
\begin{align}\label{energy==}
    \nnm{\re}_{L^{\infty}_tL^2_{x,w}}^2+\e^{-2}\tnnm{\ire}^2\ls \oo\e^{-1}\tnnm{\bre}^2+\oo^{-1},
\end{align}
and the kernel bound
\begin{align}\label{kernel==}
    \tnnm{\bre}^2\ls\e\nnm{\re}_{L^{\infty}_tL^2_{x,w}}^2+\tnnm{\ire}^2 +\e.
\end{align}
For the kernel estimate, however, the fresh terms $\e^{-1}\abs{\dobr{\ss^{B\!L},\test}}+\abs{\dbbr{\ss^{B\!L},w\cdot\nx\test}}+\e\abs{\dobr{\ss^{B\!L},\Test}}$ in \eqref{kernel 09} require a different treatment, especially for the worst part $\e^{-1}\abs{\dobr{\ss^{B\!L}_b,\test}}$: 
note that the estimate like \eqref{kernel 11} no longer applies since now we do not have $\xi|_{\p\Omega}=0$.
Instead, using integration by parts in $\mu$ and noticing that the boundary term vanishes due to 
$\int_{\s^2} \sin\phi\,\uu_1\big|_{\mu=0}\ud w = \int_{\vw\cdot\vn<0}h(t,\vx_0,\vw)(\vw\cdot\vn)\,\ud\vw=0$ for each $\vx_0\in\p\Omega$ and all $t\in\rp$ (which is precisely the compatibility condition \eqref{g-compatibility=}), we find 
\begin{align}
\e^{-1}\abs{\dobr{\ss^{B\!L}_b,\test}}&=
    \e^{-1}\Babs{\dbbr{\sin\phi\, \p_{\eta}\uu_1+\big(\uu_1-\overline{\uu_1}\big), \test}}=\babs{\dbbr{\sin\phi\,\p_{\mu}\uu_1, \test}}
    =\babs{\dbbr{\sin\phi\,\uu_1, \p_{\mn}\test}}\\
    &\ls
    \bnnm{\uu_1}_{L^2_tL^2_xL^1_w}\tnnm{\p_{\mu}\test} \ls\e^{\frac{1}{2}}\nnm{\test}_{L^2_tH^1_x}
    \ls \e^{\frac{1}{2}}\tnnm{\bre}
    \ls\oo\tnnm{\bre}^2+\oo^{-1}\e.\no
\end{align}
The rest of the terms can be bounded either in a similar way as \eqref{kernel 11}--\eqref{kernel 13}, or even more directly considering the higher order of $\e$ they possess.

Finally, by combining \eqref{energy==} and \eqref{kernel==} we have the remainder estimate
\begin{align}
    \nnm{\re}_{L^{\infty}_tL^2_{x,w}}+\e^{-\frac{1}{2}}\tnnm{\bre}+\e^{-1}\tnnm{\ire}\ls 1.
\end{align}
This leads to the proof of \eqref{main-3} and so that of Theorem~\ref{main theorem-3}.

\bigskip

\bibliographystyle{siam}
\bibliography{Reference}

\end{document}